\RequirePackage[l2tabu, orthodox]{nag} % In-hibit use of non-ams-math math-e-mat-ics markup when us-ing ams-math
\documentclass[11pt,a4paper]{article}

\title{Shifted Lanczos method for \\ quadratic forms with Hermitian matrix resolvents}
\author{Keiichi Morikuni\thanks{Faculty of Engineering, Information and Systems, University of Tsukuba, Japan. Email: morikuni@cs.tsukuba.ac.jp. The work was supported in part by the Japan Society for the Promotion of Science (Grants-in-Aid for Young Scientists (B) 16K17639) and Hattori Hokokai Foundation.}}
\date{}

\usepackage{amsmath, amsfonts, amssymb, amsthm}
\usepackage[all, warning]{onlyamsmath} % De-tect-ing and warn-ing about ob-so-lete LaTeX com-mands
\usepackage{graphicx}
\usepackage[
	bookmarks=false,
	bookmarksnumbered=true,%
	bookmarksopen=false,	
	bookmarkstype=toc,%
	pdfauthor={Keiichi Morikuni},%
	pdfdisplaydoctitle={true},
	pdfkeywords={Lanczos method quadratic form Hermitian matrix resolvent Krylov subspace moment matching model reduction},%
	pdftitle={Shifted Lanczos method for quadratic forms with Hermitian matrix resolvents},%
	setpagesize={false},%
	urlcolor={red},
	pdfstartview={FitH -32768}
]{hyperref}
\hypersetup{pdfpagemode=UseNone} 

\usepackage{tikz}
\usepackage{arydshln}

\usepackage{enumitem}
\usepackage[T1]{fontenc}
\usepackage{exscale, fix-cm, lmodern, textcomp}
\usepackage[top=1in, bottom=1in, left=1in, right=1in]{geometry}
\usepackage{algorithm, algorithmic}
\usepackage[subrefformat=parens]{subcaption}

\usepackage{algorithm, algorithmic}
\newcommand{\LINEIF}[2]{%
	\STATE\algorithmicif\ {#1}\ \algorithmicthen\ {#2}%
}

\theoremstyle{plain}
\newtheorem{theorem}{Theorem}[section]
\newtheorem*{remark}{Remark}
\newtheorem{lemma}[theorem]{Lemma}

\numberwithin{equation}{section}
\makeatletter
	
	\@addtoreset{equation}{section}
\makeatother

\makeatletter
  
  \@addtoreset{figure}{section}
\makeatother

\makeatletter
  
  \@addtoreset{algorithm}{section}
\makeatother

\makeatletter
  
  \@addtoreset{table}{section}
\makeatother

\newfont{\bg}{cmr9 scaled\magstep4}
\newcommand{\bigzerol}{\smash{\lower1.0ex\hbox{\bg 0}}}
\newcommand{\bigzerou}{\smash{\hbox{\bg 0}}}

\begin{document}
\maketitle

\begin{abstract}
% Body of abstract:
Quadratic forms of Hermitian matrix resolvents involve the solutions of shifted linear systems.
Efficient iterative solutions use the shift-invariance property of Krylov subspaces.
The Hermitian Lanczos method reduces a given vector and matrix to a Jacobi matrix (real symmetric tridiagonal matrix with positive super and sub-diagonal entries) and approximates the quadratic form using the Jacobi matrix.
This study develops a shifted Lanczos method that deals directly with the Hermitian matrix resolvent.
We derive a matrix representation of a linear operator that approximates the resolvent by solving a Vorobyev moment problem associated with the shifted Lanczos method.
We show that an entry of the Jacobi matrix resolvent can approximate the quadratic form, matching the moments.
We give a sufficient condition such that the method does not break down, an error bound, and error estimates.
Numerical experiments on matrices drawn from real-world applications compare the proposed method with previous methods and show that the proposed method outperforms well-established methods in solving some problems.
\end{abstract}
% Keywords:
%{Lanczos method; Krylov subspace; moment matching; quadratic form; Hermitian matrix; resolvent.}

\section{Introduction}
\label{sec:intro}
Consider the computation of $m$ quadratic forms
\begin{equation} \label{eq:quadform}
\boldsymbol{v}^\mathsf{H} (z_i \mathrm{I} - A)^{-1} \boldsymbol{v}, \quad i = 1, 2, \dots, m,
\end{equation}
where $\boldsymbol{v}^\mathsf{H}$ denotes the complex conjugate transpose of a vector $\boldsymbol{v} \in \mathbb{C}^n$, $A \in \mathbb{C}^{n \times n}$ is a Hermitian matrix that may be indefinite, and $z_i \in \mathbb{C}$.
Here, $z_i \mathrm{I} - A$ is assumed to be invertible.
If $z_i$ is not real, then $z_i \mathrm{I} - A$ is not Hermitian.

A straightforward approach to the quadratic form~\eqref{eq:quadform} is to solve the shifted linear systems
\begin{align}
(z_i \mathrm{I} - A) \boldsymbol{x}^{(i)} = \boldsymbol{v}
\label{eq:linsys}
\end{align}
and compute $\boldsymbol{v}^\mathsf{H} \boldsymbol{x}^{(i)}$ for $i = 1$, $2$, $\dots$, $m$.
The development of efficient solutions for shifted linear systems~\eqref{eq:linsys} with real symmetric $A$ has been on-going for two decades; these solutions include shifted Krylov subspace methods such as variants of the conjugate orthogonal conjugate gradient (COCG) method~\cite{VorstMelissen1990}, a version of the quasi-minimal residual (QMR\_SYM) method~\cite{Freund1992SISSC}, and the conjugate orthogonal conjugate residual (COCR) method~\cite{SogabeZhang2007JCAM} proposed in~\cite{TakayamaHoshiSogabeZhangFujiwara2006}, \cite{SogabeHoshiZhangFujiwara2008ETNA}, and \cite{SogabeZhang2011}, respectively.
These methods use the shift-invariance property of the Krylov subspace $\mathcal{K}_k (z_i \mathrm{I} - A, \boldsymbol{v}) = \mathcal{K}_k (A, \boldsymbol{v}) = \mathrm{span}\lbrace \boldsymbol{v}_1, A \boldsymbol{v}_1, \dots, A^{k-1} \boldsymbol{v}_1 \rbrace$ and complex symmetry of $z \mathrm{I} - A$ for efficient formulations.
A shifted Krylov subspace method forms basis vectors of the Krylov subspace~$\mathcal{K}_k (A, \boldsymbol{v})$ for a shifted matrix~$z \mathrm{I} - A$ with a particular shift~$z$ by using a short-term recurrence and determines an iterate for the linear system~$(z \mathrm{I} - A) \boldsymbol{x}_z = \boldsymbol{v}$ and associated iterates for other shifted linear systems~$(z_i \mathrm{I} - A) \boldsymbol{x}^{(i)} = \boldsymbol{v}$ simultaneously for different values of $z_i$, $i = 1$, $2$, $\dots$, $m$, without additional matrix-vector products.
Typically, shifted Krylov subspace methods take the iterative residual vectors of the shifted linear systems to be collinear to that of the seed linear system.
Here, we call a representative linear system~$(z \mathrm{I} - A) \boldsymbol{x}_z = \boldsymbol{v}$ a seed linear system.
These methods may suffer from breakdown, i.e.\ division by zero, although this rarely occurs in practice.
Previous studies on shifted CG and MINRES methods~\cite{FrommerMaass1999SISC,Meerbergen2003SIMAX} focused on real shifts.
The technique in \cite{Meerbergen2003SIMAX} improves the convergence of the methods with preconditioning under the assumption that the factorization or the use of a direct solver for the shifted symmetric matrix is performed efficiently.
Previous studies on a shifted Lanczos method~\cite{FrommerKahlLippertRittich2013SIMAX} for a Hermitian positive definite~$A$ with a complex shift~$z_i$ derived its error bound and estimates.
Extensions of the MINRES method~\cite{PaigeSaunders1975} to the shifted linear system~\eqref{eq:linsys} work with the Hermitian~$A$ and complex shift~$z_i$~\cite{GuLiu2013,SeitoHoshiYamamoto2019}.
See \cite{HoshiKawamuraYoshimiMotoyamaMisawaYamajiTodoKawashimaSogabe2020} for more shifted Krylov subspace methods.
The Pad\'{e} approximation via the Lanczos process (PVL)~\cite{FeldmannFreund1995} has attracted interest for the case~$|z| > \rho(A)$, where $\rho(\cdot)$ is the spectral radius of a square matrix, whereas our interest includes the contrasting case~$|z| < \rho(A)$.

This study focuses on exploiting the advantage of a Krylov subspace method in efficiently approximating the Hermitian matrix resolvent and directly computing the quadratic form~\eqref{eq:quadform} without solving the shifted linear systems~\eqref{eq:linsys}.
The key feature of our approach is the development of a shifted Lanczos method.
The Hermitian Lanczos method projects an original model represented by $A$ and $\boldsymbol{v}$ to a lower-order model and matches the lowest-order moments of the original model with those of the reduced model.
We show that the shifted Lanczos method retains such a property.
The Vorobyev moment problem~\cite{Vorobyev1965,Brezinski1996LAA,Strakos2009,LiesenStrakos2013} enables a concise derivation of thisto concisely derive the method.
This problem gives a matrix representation of a linear operator that represents the reduced model to approximate the resolvent~$(z \mathrm{I} - A)^{-1}$.
We show that the $(1, 1)$ entry of the Jacobi matrix resolvent can approximate the quadratic form~\eqref{eq:quadform} with eight additional operations to the Lanczos method for each shift by using a recursive formula.
Moreover, we give a sufficient condition such that the proposed method does not break down.
Breakdown may occur in the COCG, COCR, and QMR\_SYM methods, although this is rarely seen.
Furthermore, we drive an error bound and develop practical error estimates of the shifted Lanczos method.

Practical applications that can benefit from this development include chemistry and \break physics~\cite[Section~3.9]{LiesenStrakos2013}, eigensolvers using complex moments~\cite{SakuraiTadano2007}, the computation of Green's function for a many-particle Hamiltonian~\cite{YamamotoSogabeHoshiZhangFujiwara2008}, the stochastic estimation of the number of eigenvalues~\cite{MaedaFutamuraImakuraSakurai2015JSIAMLetters}, samplers for determinantal point processes~\cite{LiSraJegelka2016ICML} (see also references therein), the approximation of Markov chains in Bayesian sampling~\cite{JohndrowMattinglyMukherjeeDunson2017arXiv}, and computational quantum physics~\cite{HoshiKawamuraYoshimiMotoyamaMisawaYamajiTodoKawashimaSogabe2020}.
An extension of the single-vector case $\boldsymbol{v} \in \mathbb{R}^n$ in~\eqref{eq:quadform} for real $A \in \mathbb{R}^{n \times n}$ to the multiple-vector case $\boldsymbol{v}_1$, $\boldsymbol{v}_2$, $\dots$, $\boldsymbol{v}_\ell \in \mathbb{R}^n$, namely,
\begin{align}
V^\top (z_i \mathrm{I} - A)^{-1} V, \quad V = [\boldsymbol{v}_1, \boldsymbol{v}_2, \dots, \boldsymbol{v}_\ell] \in \mathbb{R}^{n \times \ell}, \quad i = 1, 2, \dots, m
\end{align}
can be reduced to the solutions of bilinear forms ${\boldsymbol{v}_p}^\top (z_i \mathrm{I} - A) \boldsymbol{v}_q$, $p$, $q = 1$, $2$, $\dots$, $\ell$, where $\boldsymbol{v}_p$ is the $p$th column of~$V$.
The bilinear form $\boldsymbol{v}_p (z_i \mathrm{I} - A)^{-1} \boldsymbol{v}_q$ can be further reduced to the quadratic form 
\begin{align}
\boldsymbol{v}_p (z_i \mathrm{I} - A)^{-1} \boldsymbol{v}_q = \frac{1}{4} [ \boldsymbol{s}^\top (z_i \mathrm{I} - A)^{-1} \boldsymbol{s} - \boldsymbol{t}^\top (z_i \mathrm{I} - A)^{-1} \boldsymbol{t}],
\end{align}
where $\boldsymbol{s} = \boldsymbol{v}_p + \boldsymbol{v}_q$ and $\boldsymbol{t} = \boldsymbol{v}_p - \boldsymbol{v}_q$ (cf.\ \cite[p.~114]{GolubMeurant2010}).
This type of problem arises in the analysis of dynamical systems~\cite{BaiGolub2002}.

The rest of this paper is organized as follows.
In Section{~\ref{sec:Lanczos}}, we review the Lanczos method and its moment-matching property.
In Section{~\ref{sec:shiftedLanczos}}, we describe a shifted Lanczos method for computing the quadratic forms~\eqref{eq:quadform}, its moment-matching property, and implementation; discuss related methods; and give a sufficient breakdown-free condition, error bound, and error estimates.
In Section{~\ref{sec:exp}}, we present the results of numerical experiments in which the shifted Lanczos method is compared with previous methods and illustrate the developed error estimates.
In Section{~\ref{sec:conc}}, we conclude the paper.

\section{Lanczos method} \label{sec:Lanczos}
We review the Lanczos method~\cite{Lanczos1952} for Hermitian $A \in \mathbb{C}^{n \times n}$ and $\boldsymbol{v} \in \mathbb{C}^n$.
Algorithm~\ref{alg:Lanczos} gives the procedures of the Lanczos method.
\begin{algorithm}
	\caption{Lanczos method}
	\label{alg:Lanczos}
	\begin{algorithmic}[1]
		\REQUIRE $A \in \mathbb{C}^{n \times n}$, $\boldsymbol{v} \in \mathbb{C}^n$
		\ENSURE $\boldsymbol{v}_i \in \mathbb{C}^n$, $\alpha_i \in \mathbb{R}$, $\beta_i \in \mathbb{R}$, $i = 1$, $2$, $\dots$, $k$
		\STATE $\boldsymbol{v}_1 = \boldsymbol{v} / \| \boldsymbol{v} \|$, $\boldsymbol{u} = A \boldsymbol{v}_1$, $\alpha_1 = \boldsymbol{u}^\mathsf{H} \boldsymbol{v}_1$
		\FOR{$k = 1, 2, \dots$}
		\STATE $\boldsymbol{u} = \boldsymbol{u} - \alpha_k \boldsymbol{v}_k$, $\beta_k = \| \boldsymbol{u} \|$
		\LINEIF{$\beta_k = 0$}{break}
		\STATE $\boldsymbol{v}_{k+1} = (\beta_k)^{-1} \boldsymbol{u}$, $\boldsymbol{u} = A \boldsymbol{v}_{k+1} - \beta_k \boldsymbol{v}_k$, $\alpha_{k+1} = \boldsymbol{u}^\mathsf{H} \boldsymbol{v}_{k+1}$
		\ENDFOR
	\end{algorithmic}
\end{algorithm}

Here, $\| \cdot \|$ denotes the Euclidean norm.
Denote the Lanczos decomposition of $A$ by
\begin{equation} 
A V_k = V_{k+1} T_{k+1, k},
\label{eq:Lanczosdecomp}
\end{equation}
where the columns of $V_k = [\boldsymbol{v}_1, \boldsymbol{v}_2, \dots, \boldsymbol{v}_k] \in \mathbb{C}^{n \times k}$ form an orthonormal basis of the Krylov subspace $\mathcal{K}_k (A, \boldsymbol{v}_1) = \mathrm{span}\lbrace
\boldsymbol{v}_1, A \boldsymbol{v}_1, \dots, A^{k-1} \boldsymbol{v}_1 \rbrace$ and  
\begin{align}
T_{k+1, k} & = 
\begin{bmatrix}
\alpha_1 & \beta_1 & & & \bigzerou \\
\beta_1 & \alpha_2 & \beta_2 \\
& \beta_2 & \ddots & \ddots\\
& & \ddots & & \beta_{k-1} \\
& & & \beta_{k-1} & \alpha_k \\
\bigzerol & & & & \beta_k
\end{bmatrix} \\
& = 
\begin{bmatrix}
T_{k, k} \\
\beta_k {\boldsymbol{e}_k}^\top
\end{bmatrix}
\in \mathbb{R}^{(k+1) \times k}.
\end{align}
Here, $\boldsymbol{e}_k \in \mathbb{R}^k$ is the $k$th Euclidean basis vector and $T_{k, k} \in \mathbb{R}^{k \times k}$ is the Jacobi matrix (real symmetric tridiagonal matrix with positive super and subdiagonal entries).
Then, $V_k^\mathsf{H} A V_k = T_{k, k}$ holds.

\subsection{Hamburger moment problem} \label{sec:Hamburger}
The Lanczos method projects the original model given by a Hermitian matrix~$A \in \mathbb{C}^{n \times n}$ and an initial vector~$\boldsymbol{v} \in \mathbb{C}^{n}$ to a lower-order model given by $T_{k, k} \in \mathbb{R}^{k \times k}$ and $\boldsymbol{e}_1$, matching the lowest-order moments of the original model and those of the reduced model, i.e., it approximates a given matrix~$A$ via moment matching.
A thorough derivation of the well-known moment-matching property~\eqref{eq:moment_matching} presented below is given in~\cite[Chapter~3]{LiesenStrakos2013}.
We review this for completeness.
Given a sequence of scalars $\xi_i$, $i = 0$, $1$, $2$, $\dots$, $2k-1$, the problem of finding a nondecreasing real distribution function $w^{(k)}(\lambda)$, $\lambda \in \mathbb{R}$, with $k$ points of increase such that the Riemann--Stieltjes integral is equal to the given sequence of scalars
\begin{equation}
\int_{-\infty}^\infty \lambda^i \mathrm{d} w^{(k)} (\lambda) = \xi_i, \quad i = 0, 1, \dots, 2k-1
\label{eq:reduced_moment}
\end{equation}
is called the Hamburger moment problem~\cite{Hamburger1919,Hamburger1920a,Hamburger1920b,Hamburger1921}.
The motivation to view the Lanczos method via the Hamburger moment problem instead of the Stieltjes moment problem is to enable working with indefinite Hermitian matrices with a shift.
Note that the Stieltjes moment problem goes on the positive real axis and its relation with the Lanczos method on Hermitian positive definite matrices has been well-established~\cite{LiesenStrakos2013}.

The left-hand side of~\eqref{eq:reduced_moment} is called the $i$th moment with respect to the distribution function $w^{(k)}(\lambda)$.
Set another moment 
\begin{equation}
\xi_i = \int_{-\infty}^\infty \lambda^i \mathrm{d} w (\lambda), \quad i = 0, 1, 2, \dots, 2k-1,
\label{eq:another_moment}
\end{equation}
and the distribution function with $n$ points of increase to
\begin{align}
w(\lambda) =
\begin{cases}
0, \quad & \lambda < \lambda_1, \\
\sum_{j=1}^i w_j, \quad & \lambda_i \leq \lambda < \lambda_{i+1}, \quad i = 1, 2, \dots, n-1,\\
\sum_{j=1}^n w_j = 1, \quad & \lambda_n \leq \lambda
\label{eq:measure}
\end{cases}
\end{align}
associated with weights $w_j = (\boldsymbol{v}^\mathsf{H} \boldsymbol{u}_j)^2 / \| \boldsymbol{v} \|^2$, $j = 1$, $2$, $\dots$, $n$, where $\lambda_1 < \lambda_2 < \dots < \lambda_n$ are the eigenvalues of $A$ and $\boldsymbol{u}_i$, $i=1$, $2$, $\dots$, $n$, are the corresponding eigenvectors.
Here, for clarity, we assume that the eigenvalues of $A$ are distinct without loss of generality.
Thus, the distribution function~$w(\lambda)$ is connected with the eigenpairs of $A$.
Then, we can express the moment~\eqref{eq:another_moment} as the Gauss--Christoffel quadrature and the quadratic form
\begin{align}
\int_{-\infty}^\infty \lambda^i \mathrm{d} w (\lambda) & = \sum_{j=1}^n w_j \{\lambda_j\}^i \label{eq:moment} \\
& = \boldsymbol{v}^\mathsf{H} A^i \boldsymbol{v}, \quad i = 0, 1, 2, \dots.
\end{align}
Thus, the solution of~\eqref{eq:reduced_moment} is given by
\begin{align}
w^{(k)}(\lambda) =
\begin{cases}
0, \quad & \lambda < \lambda_1^{(k)}, \\
\sum_{j=1}^i w_j^{(k)}, \quad & \lambda_i^{(k)} \leq \lambda < \lambda_{i+1}^{(k)}, \quad i = 1, 2, \dots, k-1, \\
\sum_{j=1}^k w_j^{(k)} = 1, & \lambda_k^{(k)} \leq \lambda
\end{cases}
\label{eq:reduced_measure}
\end{align}
associated with weights $w_j^{(k)} = ({\boldsymbol{e}_1}^\top \boldsymbol{u}_j^{(k)})^2$, $j = 1$, $2$, $\dots$, $k$, where $\lambda_1^{(k)} < \lambda_2^{(k)} < \dots < \lambda_k^{(k)}$ are the eigenvalues of $T_{k, k}$ and $\boldsymbol{u}_j^{(k)}$, $j = 1$, $2$, $\dots$, $k$, are the corresponding eigenvectors.
Here, the distribution function~$w^{(k)}(\lambda)$ is connected with the eigenpairs of $T_{k, k}$.
Because the Gauss--Christoffel quadrature is exact for polynomials up to degree $2k-1$
\begin{align}
\int_{-\infty}^\infty \lambda^i \mathrm{d} w^{(k)} (\lambda) & = \sum_{i=1}^k w_j^{(k)} \{\lambda_j^{(k)}\}^i \label{eq:GCexact} \\
& = {\boldsymbol{e}_1}^\top (T_{k, k})^i \boldsymbol{e}_1, \quad i = 0, 1, \dots, 2k-1,
\end{align}
the first $2k$ moments match 
\begin{align}
\boldsymbol{v}^\mathsf{H} A^i \boldsymbol{v} = (\boldsymbol{v}^\mathsf{H} \boldsymbol{v}) {\boldsymbol{e}_1}^\top (T_{k, k})^i \boldsymbol{e}_1, \quad i = 0, 1, \dots, 2k-1.
\label{eq:moment_matching}
\end{align}

\subsection{Model reduction via Vorobyev moment matching} \label{sec:Vorobyev}
We can state the problem of moment matching in the language of matrices via the Vorobyev moment problem.
To derive a linear operator $A_k$ that reduces the model order of $A$, we follow~\cite{Vorobyev1965} for the derivation (see also~\cite{Brezinski1996LAA,Strakos2009}).
Let 
\begin{align}
\begin{cases}
\boldsymbol{y}_1 = A \boldsymbol{v}, \\
\boldsymbol{y}_2 = A \boldsymbol{y}_1 (= A^2 \boldsymbol{v}), \\
\qquad \vdots \\
\boldsymbol{y}_{k-1} = A \boldsymbol{y}_{k-2} (= A^{k-1} \boldsymbol{v}), \\
\boldsymbol{y}_k = A \boldsymbol{y}_{k-1} (= A^k \boldsymbol{v}),
\end{cases},
\end{align}
for $k = 1$, $2$, $\dots$, where $\boldsymbol{v}$, $\boldsymbol{y}_1$, $\boldsymbol{y}_2$, $\dots$, $\boldsymbol{y}_k$ are assumed to be linearly independent.
Then, the Vorobyev moment problem involves determining a sequence of linear operators $A_k$ such that 
\begin{align}
\begin{cases}
\boldsymbol{y}_1 = A_k \boldsymbol{v}, \\
\boldsymbol{y}_2 = A_k \boldsymbol{y}_1 (= (A_k)^2 \boldsymbol{v}), \\
\qquad \vdots \\
\boldsymbol{y}_{k-1} = A_k \boldsymbol{y}_{k-2} (= (A_k)^{k-1} \boldsymbol{v}),\\
Q_k \boldsymbol{y}_k = A_k \boldsymbol{y}_{k-1} (= (A_k)^k \boldsymbol{v}),
\end{cases}	
\end{align}
for $k = 1$, $2$, $\dots$, where $Q_k = V_k V_k^\mathsf{H}$ is the orthogonal projector onto $\mathcal{K}_k (A, \boldsymbol{v})$.
A linear operator $A_k$ reducing the model order of $A$ is given by
\begin{align}
A_k & = Q_k A Q_k \nonumber \\
& = V_k T_{k, k} {V_k}^\mathsf{H},
\label{eq:A_k}
\end{align}
for $k = 1$, $2$, $\dots$, where the sequence $\{ A_k \}_{k \geq 0}$ is strongly convergent to $A$~\cite[Theorem~II]{Vorobyev1965} (see~\cite[Section~4.2]{Brezinski1996LAA} for the derivation of~\eqref{eq:A_k}).
Therefore, the first $2k$ moments of the reduced model match those of the original model
\begin{align}
\boldsymbol{v}^\mathsf{H} A^{i} \boldsymbol{v} & = \boldsymbol{v}^\mathsf{H} (A_k)^i \boldsymbol{v} \nonumber \\
& = (\boldsymbol{v}^\mathsf{H} \boldsymbol{v}) {\boldsymbol{e}_1}^\mathsf{H} (T_{k,k})^i \boldsymbol{e}_1, \quad i = 0, 1, \dots, 2k-1.
\label{eq:Vorobyev_moment_matching}
\end{align}
for $k = 1$, $2$, $\dots$.
We will use this property to derive a shifted Lanczos method in the next section.

\section{Shifted Lanczos method} \label{sec:shiftedLanczos}
Next, we formulate a shifted Lanczos method to approximate the resolvent $(z_i \mathrm{I} - A)^{-1}$ in \eqref{eq:quadform}.
For convenience, we omit subscript $i$ of $z_i$ if no confusion can arise.
The application of Vorobyev's method of moments to the shifted matrix $S = z \mathrm{I} - A$ and vector $\boldsymbol{v}$ gives a matrix representation of a linear operator that represents the reduced model to approximate the resolvent $(z \mathrm{I} - A)^{-1}$.
Let
\begin{align}
\begin{cases}
\boldsymbol{y}_1 = S \boldsymbol{v}, \\
\boldsymbol{y}_2 = S \boldsymbol{y}_1 (=S^2 \boldsymbol{v}), \\
\qquad \vdots \\
\boldsymbol{y}_{k-1} = S \boldsymbol{y}_{k-2} (= S^{k-1} \boldsymbol{v}), \\
\boldsymbol{y}_k = S \boldsymbol{y}_{k-1} (= S^k \boldsymbol{v}), 
\label{eq:Vorobyev_vv}
\end{cases}	
\end{align}
for $k = 1$, $2$, $\dots$, where $\boldsymbol{v}$, $\boldsymbol{y}_1$, $\boldsymbol{y}_2$, $\dots$, $\boldsymbol{y}_k$ are assumed to be linearly independent.
Then, the Vorobyev moment problem involves determining a sequence of linear operators $S_k$ such that 
\begin{align}
\begin{cases}
\boldsymbol{y}_1 = S_k \boldsymbol{v}, \\
\boldsymbol{y}_2 = S_k \boldsymbol{y}_1 (= (S_k)^2 \boldsymbol{v}), \\
\qquad \vdots \\
\boldsymbol{y}_{k-1} = S_k \boldsymbol{y}_{k-2} (= (S_k)^{k-1} \boldsymbol{v}), \\
Q_k \boldsymbol{y}_k = S_k \boldsymbol{y}_{k-1} (= (S_k)^k \boldsymbol{v}),
\end{cases}
\label{eq:Vorobyev_vvv}
\end{align}
for $k = 1$, $2$, $\dots$, where $Q_k = V_k V_k^\mathsf{H}$ is the orthogonal projector onto $\mathcal{K}_k (S, \boldsymbol{v})$.
We first solve the problem for the linear operator $S_k$.
Equations~\eqref{eq:Vorobyev_vvv} can be written as 
\begin{equation}
\boldsymbol{y}_i = (S_k)^i \boldsymbol{v}, \quad i = 1, 2, \dots, k-1, \quad Q_k \boldsymbol{y}_k = (S_k)^k \boldsymbol{v}
\end{equation}
for $k = 1$, $2$, $\dots$.
An arbitrary vector $\boldsymbol{u} \in \mathcal{K}_k(S, \boldsymbol{v})$ is expanded as
\begin{equation}
\boldsymbol{u} = \sum_{i=0}^{k-1} a_i \boldsymbol{y}_i, \quad a_i \in \mathbb{C},
\end{equation}
where $\boldsymbol{y}_0 = \boldsymbol{v}$.
Multiplying both sides by $S$ gives
\begin{equation}
S \boldsymbol{u} = \sum_{i=0}^{k-2} a_i S^{i+1} \boldsymbol{v} + a_{k-1} \boldsymbol{y}_k.
\end{equation}
Projecting this onto $\mathcal{K}_k (S, \boldsymbol{v}) = \mathcal{K}_k (A, \boldsymbol{v})$ (shift-invariance property) gives
\begin{align}
Q_k S \boldsymbol{u} & = \sum_{i=0}^{k-2} a_i (S_k)^{i+1} \boldsymbol{v} + a_{k-1} (S_k)^k \boldsymbol{v} \\
& = \sum_{i=0}^{k-1} a_i (S_k)^{i+1} \boldsymbol{v} \\
& = \sum_{i=0}^{k-1} a_i S_k \boldsymbol{y}_i \\
& = S_k \boldsymbol{u}.
\label{eq:proj_QSu}
\end{align}
Here, the first equality is due to equations \eqref{eq:Vorobyev_vv}, \eqref{eq:Vorobyev_vvv}, and 
\begin{align}
Q_k (S_k)^{i+1} \boldsymbol{v} & = Q_k S^{i+1} \boldsymbol{v} \\
& = S^{i+1} \boldsymbol{v} \in \mathcal{K}_k (S, \boldsymbol{v}), \quad i = 0, 1, \dots, k-2.
\end{align}
Hence, \eqref{eq:proj_QSu} shows that $Q_k S = S_k$ on $\mathcal{K}_k (S, \boldsymbol{v})$.
Because $Q_k \boldsymbol{w} \in \mathcal{K}_k (S, \boldsymbol{v})$ for any vector $\boldsymbol{w} \in \mathbb{C}^n$, we can obtain the expression
\begin{equation}
S_k = Q_k S Q_k
\label{eq:Sk}
\end{equation}
by extending the domain to the whole space $\mathbb{C}^n$.
Note that the sequence $\{ S_k \}_{k \geq 0}$ is strongly convergent to $S$~\cite[Theorem~II]{Vorobyev1965}.
The expression~\eqref{eq:Sk} can be obtained from the shifted Lanczos decomposition~\cite[Lemma~2.1 (i)]{FrommerKahlLippertRittich2013SIMAX}
\begin{equation}
(z \mathrm{I} - A) V_k = V_{k+1} \left(z 
\begin{bmatrix}
\mathrm{I} \\
\boldsymbol{0}^\top
\end{bmatrix} - T_{k+1, k}
\right).
\label{eq:shiftedLanczos}
\end{equation}
Multiplying this by $V_k^\mathsf{H}$ gives
\begin{equation}
V_k^\mathsf{H} S V_k = T_k^<.
\end{equation}
This gives the orthogonally projected restriction 
\begin{align}
S_k & = V_k {V_k}^\mathsf{H} S V_k V_k^\mathsf{H} \\
& = V_k T_k^< V_k^\mathsf{H}.
\end{align}
However, this way does not give the insight of strong convergence.

By using the expression~$S_k$, we show that the moments of the original model with $S$ and $\boldsymbol{v}$ and those of the reduced model with $T_k^< = z \mathrm{I} - T_{k, k}$ and $\boldsymbol{e}_1$ match.
By using the moment-matching property~\eqref{eq:Vorobyev_moment_matching} and the binomial formula, we have
\begin{align}
\boldsymbol{v}^\mathsf{H} S^i \boldsymbol{v} = \boldsymbol{v}^\mathsf{H} (z \mathrm{I} - A_k)^i \boldsymbol{v}, \quad i = 0, 1, \dots, 2k-1,
\end{align}
Therefore, the reduced model matches the first $2k$ moments of the original model
\begin{align}
\boldsymbol{v}^\mathsf{H} S^i \boldsymbol{v} & = \boldsymbol{v}^\mathsf{H} (S_k)^i \boldsymbol{v} \\
& = (\boldsymbol{v}^\mathsf{H} \boldsymbol{v}) \boldsymbol{e}_1^\top (T_k^<)^i \boldsymbol{e}_1, \quad i = 0, 1, \dots, 2k-1.
\end{align}
Note that the left-hand side is equal to
\begin{equation}
\sum_{j=0}^i w_j (z-\lambda_j)^i = \int_{-\infty}^\infty (z-\lambda)^i \mathrm{d} w(\lambda)
\end{equation}
with distribution function~\eqref{eq:measure} and the right-hand side is equal to
\begin{equation}
\sum_{j=0}^i w_j (z-\lambda_j^{(k)})^i = \int_{-\infty}^\infty (z-\lambda)^i \mathrm{d} w^{(k)}(\lambda)
\end{equation}
with distribution function~\eqref{eq:reduced_measure} (cf.\ \cite[Chapter~15]{Szego1959}).
Thus, the Hamburger moment problem (Section~\ref{sec:Hamburger}) gives these connections that carry over the expressions of moments~\eqref{eq:moment} and the exactness of the quadrature~\eqref{eq:GCexact} to the shifted case, whereas Vorobyev's method of moments translates these connections in the language of matrices.

Consider the approximation of $\boldsymbol{v}^\mathsf{H} S^{-1} \boldsymbol{v}$.
Because the matrix representation of the inverse of the reduced-order operator $S_k$ restricted onto $\mathcal{K}_k (S, \boldsymbol{v})$~\cite[p.~79]{HoffmanKunze1971} is given by
\begin{equation}
{S_k}^{-1} = V_k (T_k^<)^{-1} {V_k}^\mathsf{H},
\label{eq:Sk_inv}
\end{equation}
an approximation of $\boldsymbol{v}^\mathsf{H} S^{-1} \boldsymbol{v}$ is given by $\boldsymbol{v}^\mathsf{H} (S_k)^{-1} \boldsymbol{v}$.
Therefore, we obtain
\begin{equation}
\boldsymbol{v}^\mathsf{H} (S_k)^{-1} \boldsymbol{v} = (\boldsymbol{v}^\mathsf{H} \! \boldsymbol{v}) {\boldsymbol{e}_1}^\top (T_k^<)^{-1} \boldsymbol{e}_1 \equiv L_k.
\label{eq:appox_quadform}
\end{equation}

\subsection{Implementation}
The quantity~$L_k$ in \eqref{eq:appox_quadform} is the $(1, 1)$ entry of the resolvent~$(T_k^<)^{-1}$ of a successively enlarging Jacobi matrix~$T_{k, k}$.
An efficient recursive formula for computing such an entry was developed in~\cite[Section~3.4]{GolubMeurant2010}.
The formula is given by
\begin{equation}
L_{k+1} = L_k + c_{k+1} \pi_{k+1}, \quad k = 1, 2, \dots,
\end{equation}
starting with $c_1 = 1$, $\delta_1 = z - \alpha_1$, and $\pi_1 = 1 / \delta_1$, where 
\begin{align}
t_k & = (\beta_k)^2 \pi_k, \\
\delta_{k+1} & = z - \alpha_{k+1} - t_k, \label{eq:delta} \\
\pi_{k+1} & = 1 / \delta_{k+1}, \\
c_{k+1} & = c_k t \pi_k.
\end{align}
We summarize the procedures for approximating quadratic forms~\eqref{eq:quadform} in Algorithm~\ref{alg:shiftedLanczos}.
Here, we denote $L_k^{(i)} = (\boldsymbol{v}^\mathsf{T} \boldsymbol{v}) \boldsymbol{e}_1^\mathsf{T} (z_i \mathrm{I} - T_{k,k})^{-1} \boldsymbol{e}_1$.
Quantities denoted with the superscript~$(i)$ correspond to the shift~$z_i$.
The difference from Algorithm~\ref{alg:Lanczos} is the addition of Lines~2 and 7.
In particular, when $A$ is a real symmetric matrix and $\boldsymbol{v}$ is a real vector in Algorithm~\ref{alg:shiftedLanczos}, only real arithmetic is needed to compute Lines~1, 4, and 6, whereas Lines~2 and 7 require complex arithmetic in general.
Note that $\alpha_{k+1}$ is a real number in theory.
However, when $A$ is not real but complex Hermitian, due to rounding error in finite precision arithmetic, the imaginary part of $\alpha_{k+1}$ may grow and affect the accuracy.
Therefore, it is recommended to explicitly set the real part of the numerically computed $\alpha_{k+1}$ to its value.

\begin{algorithm}
	\caption{Shifted Lanczos method for quadratic forms}
	\label{alg:shiftedLanczos}
	\begin{algorithmic}[1]
		\REQUIRE $A \in \mathbb{C}^{n \times n}$, $\boldsymbol{v} \in \mathbb{C}^n$, $z_i \in \mathbb{C}$, $i=1$, $2$, $\dots$, $m$
		\ENSURE	$L_k^{(i)} \in \mathbb{C}$, $i=1$, $2$, $\dots$, $m$
		\STATE $\boldsymbol{v}_1 = \boldsymbol{v} / \| \boldsymbol{v} \|$, $\boldsymbol{s}_1 = A \boldsymbol{v}_1$, $\alpha_1 = \boldsymbol{u}^\mathsf{H} \boldsymbol{v}_1$
		\STATE  $c_1^{(i)} = \boldsymbol{v}^\top \boldsymbol{v}$, $\delta_1^{(i)} = z_i - \alpha_1$, $\pi_1^{(i)} = 1 / \delta_1^{(i)}$, $L_1^{(i)} = c_1^{(i)} / (z_i -\alpha_1)$, $i = 1$, $2$, $\dots$, $m$
		\FOR{$k = 1, 2, \dots$, until convergence}
		\STATE $\boldsymbol{t}_k = \boldsymbol{s}_k - \alpha_k \boldsymbol{v}_k$, $\beta_k = \| \boldsymbol{t}_k \|$
		\LINEIF{$\beta_k = 0$}{break}
		\STATE $\boldsymbol{v}_{k+1} = (\beta_k)^{-1} \boldsymbol{t}_k$, $\boldsymbol{s}_{k+1} = A \boldsymbol{v}_{k+1} - \beta_k \boldsymbol{v}_k$, $\alpha_{k+1} = \boldsymbol{s}_{k+1}^\mathsf{H} \boldsymbol{v}_{k+1}$
		\STATE $t_k^{(i)} = (\beta_k)^2 \pi_k^{(i)}$, $\delta_{k+1}^{(i)} = z_i - \alpha_{k+1} - t_k^{(i)}$, $\pi_{k+1}^{(i)} = 1 / \delta_{k+1}^{(i)}$, $c_{k+1}^{(i)} = c_k^{(i)} t_k^{(i)} \pi_k^{(i)}$, $L_{k+1}^{(i)} = L_k^{(i)} + c_{k+1}^{(i)} \pi _{k+1}^{(i)}$, $i=1$, $2$, $\dots$, $m$
		\ENDFOR
	\end{algorithmic}
\end{algorithm}

We compare the shifted Lanczos method with related methods in terms of computational cost.
Algorithms~\ref{alg:shiftedCOCG}, \ref{alg:shiftedCOCR}, and \ref{alg:shiftedMINRES} give simple modifications of the shifted COCG, COCR, and MINRES methods, respectively, for computing quadratic forms~\eqref{eq:quadform}.
Here, $z_s$ is the seed shift.
The modifications are given by applying $\boldsymbol{v}^\top$ to the $k$th iterate $\boldsymbol{x}_k$ and associated vectors.
They produce approximations $G_k^{(i)}$, $R_k^{(i)}$, and $M_k^{(i)}$, respectively, to the quadratic forms~\eqref{eq:quadform}.
In the shifted COCG and COCR methods, if a satisfactory seed iterate is obtained, one can use the seed switching technique~\cite{YamamotoSogabeHoshiZhangFujiwara2008} to choose a different shift as the seed shift~$z_s$.
Similarly to the shifted MINRES method~\cite[Section~2.3]{SeitoHoshiYamamoto2019}, the shifted Lanczos method does not need such a seed switching technique.

\begin{remark}
	The modifications applied to the shifted COCG, COCR, and MINRES methods for computing quadratic forms to derive Algorithms~\ref{alg:shiftedCOCG}, \ref{alg:shiftedCOCR}, and \ref{alg:shiftedMINRES}, respectively, can also be applied to the shifted CG method to compute quadratic forms.
	However, such a shifted modification of the CG method~\cite[Section~2.2]{HoshiKawamuraYoshimiMotoyamaMisawaYamajiTodoKawashimaSogabe2020} is mathematically equivalent to the shifted Lanczos method for Hermitian positive definite~$z \mathrm{I} - A$, whereas it is not mathematically equivalent to the shifted Lanczos method in general.
	Such a shifted CG method may not work on the case where $z$ is not real, because the shifted coefficient matrix~$z \mathrm{I} - A$ is not Hermitian for $z \in \mathbb{C} \backslash \mathbb{R}$.
	With the connection between the tridiagonal entries of the Jacobi matrix and the scalar coefficients of vectors in the CG method~\cite[Section~6.7]{Saad2003}, we may formulate the shifted CG method for quadratic forms, which is mathematically equivalent to the shifted Lanczos method for quadratic forms.
	See \cite[Section~9.6]{Meurant2006} for different aspects of the shifted CG method.
\end{remark}

For simplicity of comparison, we count basic scalar, vector, and matrix operations.
The Lanczos method (Algorithm~\ref{alg:Lanczos}) needs one vector scale (\texttt{scale}), one dot product (\texttt{dot}), one vector norm (\texttt{norm}), two scalar--vector additions (\texttt{axpy}) and one matrix--vector product (\texttt{matvec}) per iteration.
Table~\ref{tbl:op} gives the number of basic vector and matrix operations of the shifted methods.
Table~\ref{tbl:scalar_op} gives the number of scalar operations for each shift $z_i$ per iteration.
These tables show that in terms of the cost per iteration, the shifted Lanczos method is the cheapest of the methods compared.

\begin{algorithm}
	\caption{Shifted COCG method for quadratic forms}
	\label{alg:shiftedCOCG}
	\begin{algorithmic}[1]
		\REQUIRE $A \in \mathbb{R}^{n \times n}$, $\boldsymbol{v} \in \mathbb{C}^n$, $z_s \in \mathbb{C}$, $z_i \in \mathbb{C}$, $i=1$, $2$, $\dots$, $m$
		\ENSURE	$G_k^{(i)} \in \mathbb{C}$, $i=1$, $2$, $\dots$, $m$
		\STATE $\alpha_{-1} = 1$, $\beta_{-1} = 0$, $\boldsymbol{p}_{-1} = \boldsymbol{0}$, $\boldsymbol{r}_0 = \boldsymbol{v}$, $\boldsymbol{p}_0 = \boldsymbol{r}_0$
		\STATE $\pi_{-1}^{(i)} = \pi_0^{(i)} = 1$, $p_0^{(i)} = \boldsymbol{v}^\mathsf{H} \boldsymbol{r}_0$, $G_0^{(i)} = 0$, $i = 1$, $2$, $\dots$, $m$
		\FOR{$k = 1, 2, \dots$, until convergence}
		\LINEIF{${\boldsymbol{p}_{k-1}}^\top (z_s \mathrm{I} - A) \boldsymbol{p}_{k-1} = 0$ or ${\boldsymbol{r}_{k-1}}^\top \boldsymbol{r}_{k-1} = 0$}{switch the seed}
		\STATE $\alpha_{k-1} = ({\boldsymbol{r}_{k-1}}^\top \boldsymbol{r}_{k-1}) / ({\boldsymbol{p}_{k-1}}^\top (z_s \mathrm{I} - A) \boldsymbol{p}_{k-1})$, $\boldsymbol{r}_k = \boldsymbol{r}_{k-1} - \alpha_{k-1} (z_s \mathrm{I} - A) \boldsymbol{p}_{k-1}$, $\beta_{k-1} = ({\boldsymbol{r}_k}^\top \boldsymbol{r}_k) / ({\boldsymbol{r}_{k-1}}^\top \boldsymbol{r}_{k-1})$, $r_k = \boldsymbol{v}^\mathsf{H} \boldsymbol{r}_k$, $\boldsymbol{p}_k = \boldsymbol{r}_k + \beta_{k-1} \boldsymbol{p}_{k-1}$ 
		\FOR{$i = 1$, $2$, $\dots$, $m$}
		\STATE $\pi_k^{(i)} = [ 1 + \alpha_{k-1} (z_i - z_s) + (\beta_{k-2}/\alpha_{k-2}) \alpha_{k-1} ] \pi_{k-1}^{(i)} - (\beta_{k-2}/\alpha_{k-2}) \alpha_{k-1} \pi_{k-2}^{(i)}$
		\LINEIF{$\pi_k^{(i)} = 0$}{output $G_{k-1}^{(i)}$}
		\STATE $\alpha_{k-1}^{(i)} = (\pi_{k-1}^{(i)} / \pi_k^{(i)}) \alpha_{k-1}$, $G_k^{(i)} = G_{k-1}^{(i)} + \alpha_{k-1}^{(i)} p_{k-1}^{(i)}$, $\beta_{k-1}^{(i)} = (\pi_{k-1}^{(i)} / \pi_k^{(i)})^2 \beta_{k-1}$, $p_k^{(i)} = r_k / \pi_k^{(i)} + \beta_{k-1}^{(i)} p_{k-1}^{(i)}$
		\ENDFOR
		\ENDFOR
	\end{algorithmic}
\end{algorithm}

\begin{algorithm}
	\caption{Shifted COCR method for quadratic forms}
	\label{alg:shiftedCOCR}	
	\begin{algorithmic}[1]
		\REQUIRE $A \in \mathbb{R}^{n \times n}$, $\boldsymbol{v} \in \mathbb{C}^n$, $z_s \in \mathbb{C}$, $z_i \in \mathbb{C}$, $i=1$, $2$, $\dots$, $m$
		\ENSURE	$R_k^{(i)} \in \mathbb{C}$, $i=1$, $2$, $\dots$, $m$
		\STATE $\alpha_{-1} = 1$, $\beta_{-1} = 0$, $\boldsymbol{q}_{-1} = \boldsymbol{0}$, $\boldsymbol{r}_0 = \boldsymbol{v}$, $r_0 = \boldsymbol{v}^\mathsf{H} \boldsymbol{r}_0$
		\STATE $p_{-1}^{(i)} = 0$, $\pi_{-1}^{(i)} = \pi_0^{(i)} = 1$, $p_0^{(i)} = \boldsymbol{v}^\mathsf{H} \boldsymbol{r}_0$, $R_0^{(i)} = 0$, $i = 1$, $2$, $\dots$, $m$
		\FOR{$k = 1, 2, \dots$, until convergence}
		\STATE $\boldsymbol{q}_{k-1} = (z_s \mathrm{I} - A) \boldsymbol{r}_{k-1} + \beta_{k-2} \boldsymbol{q}_{k-2}$
		\LINEIF{${\boldsymbol{q}_{k-1}}^\top \boldsymbol{q}_{k-1} = 0$}{switch the seed}
		\STATE $\alpha_{k-1} = [{\boldsymbol{r}_{k-1}}^\top (z_s \mathrm{I} - A) \boldsymbol{r}_{k-1}] / ({\boldsymbol{q}_{k-1}}^\top \boldsymbol{q}_{k-1})$
		\FOR{$i = 1$, $2$, $\dots$, $m$}
		\STATE  $\pi_k^{(i)} = (1 + (\beta_{k-2} / \alpha_{k-2}) \alpha_{k-1} + \alpha_{k-1} (z_i - z_s)) \pi_{k-1}^{(i)} - (\beta_{k-2} / \alpha_{k-2}) \alpha_{k-1} \pi_{k-2}^{(i)}$
		\LINEIF{$\pi_k^{(i)} = 0$}{output $R_{k-1}^{(i)}$}
		\STATE $\beta_{k-2}^{(i)} = (\pi_{k-2}^{(i)} / \pi_{k-1}^{(i)})^2 \beta_{k-2}$, $\alpha_{k-1}^{(i)} = (\pi_{k-1}^{(i)} / \pi_k^{(i)}) \alpha_{k-1}$, $p_{k-1}^{(i)} = r_{k-1} / \pi_{k-1}^{(i)} + \beta_{k-2}^{(i)} p_{k-2}^{(i)}$, $R_k^{(i)} = R_{k-1}^{(i)} + \alpha_{k-1}^{(i)} p_{k-1}^{(i)}$
		\ENDFOR
		\STATE $\boldsymbol{r}_k = \boldsymbol{r}_{k-1} - \alpha_{k-1} \boldsymbol{q}_{k-1}$, $r_k = \boldsymbol{v}^\mathsf{H} \boldsymbol{r}_k$
		\LINEIF{${\boldsymbol{r}_{k-1}}^\top (z_s \mathrm{I} - A) \boldsymbol{r}_{k-1} = 0$}{switch the seed}		
		\STATE $\beta_{k-1} = [{\boldsymbol{r}_k}^\top (z_s \mathrm{I} - A) \boldsymbol{r}_k] / [{\boldsymbol{r}_{k-1}}^\top (z_s \mathrm{I} - A) \boldsymbol{r}_{k-1}]$
		\ENDFOR
	\end{algorithmic}
\end{algorithm}

{
	\begin{algorithm}
		\caption{Shifted MINRES method for quadratic forms}
		\label{alg:shiftedMINRES}
		\begin{algorithmic}[1]
			\REQUIRE $A \in \mathbb{C}^{n \times n}$, $\boldsymbol{v} \in \mathbb{C}^n$, $\boldsymbol{x}_0 \in \mathbb{C}^n$, $z_s \in \mathbb{C}$, $z_i \in \mathbb{C}$, $i=1$, $2$, $\dots$, $m$
			\ENSURE	$M_k^{(i)} \in \mathbb{C}$, $i=1$, $2$, $\dots$, $m$
			\STATE $\beta_0 = 0$, $\boldsymbol{q}_0 = \boldsymbol{0}$, $\boldsymbol{r}_0 = \boldsymbol{v} - (z_s \mathrm{I} - A) \boldsymbol{x}_0$, $\boldsymbol{q}_1 = {\| \boldsymbol{r}_0 \|}^{-1} \boldsymbol{r}_0$, $q_1 = \boldsymbol{v}^\mathsf{H} \boldsymbol{q}_1$
			\STATE $f_1^{(i)} = 1$, $p_{-1}^{(i)} = p_0^{(i)} = 0$, $M_0^{(i)} = 0$, $i = 1$, $2$, $\dots$, $m$
			\FOR{$k = 1, 2, \dots$, until convergence}
			\STATE $\boldsymbol{s}_k = A \boldsymbol{q}_k - \beta_{k-1} \boldsymbol{q}_{k-1}$, $\alpha_k = {\boldsymbol{s}_k}^\mathsf{H} \boldsymbol{q}_k$, $\boldsymbol{t}_k = \boldsymbol{s}_k - \alpha_k \boldsymbol{q}_k$, $\beta_k = \| \boldsymbol{t}_k \|$, $q_k = \boldsymbol{v}^\mathsf{H} \boldsymbol{q}_k$
			\FOR{$i = 1$, $2$, $\dots$, $m$}
			\STATE $r_{k-2, k}^{(i)} = 0$, $r_{k-1, k}^{(i)} = \beta_{k-1}$, $r_{k, k}^{(i)} = z_i - \alpha_k$
			\LINEIF{$k \geq 3$}{update $[ r_{k-2, k}^{(i)}, r_{k-1, k}^{(i)} ]^\top= G_{k-2}^{(i)} [ r_{k-2, k}^{(i)}, r_{k-1, k}^{(i)} ]^\top$}
			\LINEIF{$k \geq 2$}{update $[ r_{k-1, k}^{(i)}, r_{k, k}^{(i)} ]^\top = G_{k-1}^{(i)} [ r_{k-1, k}^{(i)}, r_{k, k}^{(i)}]^\top$}
			\STATE Compute $G_k^{(i)} = \left[ \begin{smallmatrix} c_k^{(i)} & \bar{s}_k^{(i)} \\ -s_k^{(i)} & \bar{c}_k^{(i)} \end{smallmatrix}\right]$ and update $r_{k, k}^{(i)}$ such that $[ r_{k, k}^{(i)}, 0 ]^\top = G_k^{(i)} [ r_{k, k}^{(i)}, \beta_k ]^\top$, $|c_k^{(i)}|^2 + |s_k^{(i)}|^2 = 1$, $c_k^{(i)}$, $s_k^{(i)} \in \mathbb{C}$.
			\STATE $p_k^{(i)} = (r_{k, k}^{(i)})^{-1} (q_k	 - r_{k-2, k}^{(i)} p_{k-2}^{(i)} - r_{k-1, k}^{(i)} p_{k-1}^{(i)})$, $M_k^{(i)} = M_{k-1}^{(i)} + \| \boldsymbol{r}_0 \| c_k^{(i)} f_k^{(i)} p_k^{(i)}$, $f_{k+1}^{(i)} = -\bar{s}_k^{(i)} f_k^{(i)}$
			\ENDFOR
			\STATE $\boldsymbol{q}_{k+1} = (\beta_k)^{-1} \boldsymbol{t}_k$
			\ENDFOR
		\end{algorithmic}
	\end{algorithm}
	
	\setlength{\tabcolsep}{10pt}
	\begin{table}[t]
		\centering
		\scriptsize
		\caption{Basic vector and matrix operations}
		\begin{tabular}{lrrrrr}
			\hline\noalign{\smallskip}
			Method & \texttt{scale} & \texttt{dot} & \texttt{norm} & \texttt{axpy} & \texttt{matvec} \\
			\noalign{\smallskip}\hline\noalign{\smallskip}
			Shifted Lanczos & 1 & 1 & 1 & 2 & 1 \\
			Shifted COCG & 1 & 1 & 1 & 2 & 1 \\
			Shifted COCR & 0 & 3 & 0 & 2 & 1 \\
			Shifted MINRES & 1 & 2 & 1 & 2 & 1 \\	
			\noalign{\smallskip}\hline
		\end{tabular}
		\label{tbl:op}
		\begin{minipage}{0.76\hsize}
			Method: name of the method, \texttt{scale}: vector scale, \texttt{dot}: dot product, \texttt{norm}: vector norm, \texttt{axpy}: scalar--vector addition, \texttt{matvec}: matrix--vector product.
		\end{minipage}
		
		\caption{Scalar operations for each shift $z_i$ per iteration}
		\begin{tabular}{lrrrrr}
			\hline\noalign{\smallskip}
			Method & $+$ & $\times$ & $/$ & $\sqrt{}$ & Total\\
			\noalign{\smallskip}\hline\noalign{\smallskip}
			Shifted Lanczos & 3 & 4 & 1 & 0 & 8 \\
			Shifted COCG & 6 & 9 & 3 & 0 & 18 \\
			Shifted COCR & 6 & 8 & 3 & 0 & 17 \\
			Shifted MINRES & 11 & 20 & 3 & 1 & 35 \\
			\noalign{\smallskip}\hline
		\end{tabular}
		\label{tbl:scalar_op}
		\begin{minipage}{0.64\hsize}
			Method: name of the method, $+$: addition, $\times$: multiplication, $/$: division, $\sqrt{}$: square root, Total: total number of operations.
		\end{minipage}
	\end{table}
}

\subsection{Breakdown-free condition}
A breakdown resulting from division by zero for $\delta_{k+1}^{(i)} = 0$ may occur in Line~7 of Algorithm~\ref{alg:shiftedLanczos} before a solution is obtained.
Therefore, we give a sufficient condition such that the shifted Lanczos method does not break down.
Let $|T_0^<| = 1$ for convenience, where $|\cdot|$ denotes the determinant of a matrix.
For convenience, we omit superscript $(i)$ from the quantities given in Algorithm~\ref{alg:shiftedLanczos} and prepare lemmas on shifted versions of well-known properties of the Jacobi matrix.

\begin{lemma} \label{lm:zeroG}
	Let $T_k^<$ and $\delta_k$ be defined as above.
	Assume that $|T_i^<| \ne 0$ holds for $i = 1$, $2$, $\dots$, $k$, $k \in \mathbb{Z}_{>0}$.
	If $\delta_{k+1} = 0$ for $k \in \mathbb{Z}_{>0}$, then we have $|T_{k+1}^<| = 0$.
\end{lemma}
\begin{proof}
	The condition $\delta_{k+1} = 0$ gives $t_k = z - \alpha_{k+1}$.
	It follows from \cite[Lemma~3.2]{GolubMeurant2010} 
	\begin{align}
	|T_{k+1}^<| = (z-\alpha_{k+1}) |T_k^<| - (\beta_k)^2 |T_{k-1}^<|, \quad k > 0
	\end{align}	
	and \cite[Section~3.3]{GolubMeurant2010} 
	\begin{equation}
	\delta_{k+1} = \frac{|T_{k+1}^<|}{|T_k^<|}, \quad k > 0
	\label{eq:delta_k}
	\end{equation}
	that
	\begin{align}
	|T_{k+1}^<| & = |T_k^<| \left( z - \alpha_{k+1} - (\beta_k)^2 \frac{|T_{k-1}^<|}{|T_k^<|} \right) \\
	& = |T_k^<| \left( z - \alpha_{k+1} - t_k \right) \\
	& = 0.
	\end{align}
	\qed
\end{proof}

Then, the breakdown-free condition is described as follows.

\begin{theorem} \label{th:breakdown_free}
	Let $\delta_k$ be defined as above.
	Let $\lambda_1$ and $\lambda_n$ be the smallest and largest eigenvalues of $A$, respectively.
	If $z \in \mathbb{C}$ satisfies $z \not \in [\lambda_1, \lambda_n]$, then $\delta_{k} \ne 0$ holds for $k \in \mathbb{Z}_{> 0}$.
\end{theorem}
\begin{proof}
	From the interlacing property of eigenvalues~\cite[Theorem~4.3.17]{HornJohnson2013}, $T_{k, k}$ does not have an eigenvalue equal to $z$.
	Hence, $|T_k^<| \ne 0$ holds.
	From Lemma~\ref{lm:zeroG}, the assertion holds.
	\qed
\end{proof}

Theorem~\ref{th:breakdown_free} shows that the shifted Lanczos method does not break down whenever each $z_i \in \mathbb{C}$ satisfies $z_i \not \in [\lambda_1, \lambda_n]$.
The condition $z \not \in [\lambda_1, \lambda_n]$ in Theorem~\ref{th:breakdown_free} is not necessarily equivalent to assuming that $z \mathrm{I} - A$ is positive or negative definite because $z \in \mathbb{C}$.
The shifted MINRES methods~\cite{GuLiu2013,SeitoHoshiYamamoto2019} do not break down for nonsingular $z \mathrm{I} - A$, whereas the shifted COCG and COCR methods may break down.

Projection methods for symmetric or Hermitian eigenproblems~\cite{SakuraiSugiura2003JCAM,Polizzi2009} typically take shift points for quadrature points in a circle and can circumvent taking quadrature points on the real line.
Therefore, the shifted Lanczos method can avoid breakdown when applied to quadratic forms in these methods with particular choices of quadrature points.

\subsection{Convergence bound}
The approximation of the quadratic form~$\boldsymbol{v}^\mathsf{T} (z \mathrm{I} - A)^{-1} \boldsymbol{v}$ by using the shifted Lanczos method can be viewed as a way of solving the shifted linear system~$(z \mathrm{I} - A) \boldsymbol{x} = \boldsymbol{v}$ by using the same method and computing $\boldsymbol{v}^\mathsf{T} \boldsymbol{x}_*$, where $\boldsymbol{x}_* = S^{-1} \boldsymbol{v}$.
Concretely, to determine the $k$th iterate~$\boldsymbol{x}_k$ of the Lanczos method for the linear system~$S \boldsymbol{x} = \boldsymbol{v}$, the method imposes the Galerkin condition
\begin{align}
\boldsymbol{v} - S \boldsymbol{x}_k \perp \mathcal{K}_k (A, \boldsymbol{v}).
\end{align}
The iterate is the same as the CG iterate if $S$ is Hermitian positive definite; otherwise, they may be different.
The absolute error of the shifted Lanczos method for quadratic forms for the $k$th iteration is 
\begin{align}
\varepsilon_k = | L_k - \boldsymbol{v}^\mathsf{H} S^{-1} \boldsymbol{v} |
\label{eq:absolute_error}
\end{align}
and its upper bound is given by using the Cauchy-Schwarz inequality
\begin{align}
\varepsilon_k & = | \boldsymbol{v}^\mathsf{H} ( \boldsymbol{x}_k - S^{-1} \boldsymbol{v} ) | \\
& \leq \| \boldsymbol{v} \| \| \boldsymbol{x}_k - \boldsymbol{x}_* \|.
\end{align}
For the Hermitian positive definite linear systems~$A \boldsymbol{x} = \boldsymbol{b}$, $\boldsymbol{b} \in \mathbb{C}^n$, a well-known upper bound~\cite{Kaniel1966} of the Lanczos method is related to the $A$-norm and a recent upper bound depends on the distribution of the eigenvalues~\cite[Theorem~B.1]{MuscoMuscoSidford2018SODA}.

The following assertions give an error bound of the shifted Lanczos method for quadratic forms.

\begin{theorem} \label{th:error_bound}
	Let $A \in \mathbb{C}^{n \times n}$ be a Hermitian matrix, $z \in \mathbb{C}$, $S = z \mathrm{I} - A$, $\boldsymbol{v} \in \mathbb{C}^n$, and $\mathbb{P}_k$ be the set of all polynomials with degree less than $k$.
	Then, the absolute error of the shifted Lanczos method for the quadratic form~$\boldsymbol{v}^\mathsf{H} S^{-1} \boldsymbol{v}$ for the $k$th iteration satisfies
	\begin{align}
	\varepsilon_k \leq 2 \tau_k \| \boldsymbol{v} \|^2.
	\label{eq:errnrm}
	\end{align}
	where 
	\begin{align}
	\tau_k = \min_{p \in \mathbb{P}_k} \max_{t \in [\lambda_1, \lambda_n]} | p(z-t) - (z-t)^{-1} |.
	\label{eq:tau}
	\end{align}
\end{theorem}
\begin{proof}
	If $\boldsymbol{x}_*$ and $\boldsymbol{x}_k$ are defined as above, then the error norm has the expression
	\begin{align}
	\| \boldsymbol{x}_k - \boldsymbol{x}_* \| & = \| V_k ( T_{k, k}^< )^{-1} ( \| \boldsymbol{v} \| \boldsymbol{e}_1 ) - S^{-1} \boldsymbol{v} \| \\
	& = \| V_k ( T_{k, k}^< )^{-1} \boldsymbol{e}_1 - S^{-1} \boldsymbol{v}_1 \| \| \boldsymbol{v} \|.
	\end{align}
	For any polynomial $p \in \mathbb{P}_k$, the first factor of the last quantity is bounded as 
	\begin{align}
	& \| V_k ( T_{k, k}^< )^{-1} \boldsymbol{e}_1 - S^{-1} \boldsymbol{v}_1 \| \\
	\leq & \| V_k [ p(T_{k, k}^<) - ( T_{k, k}^< )^{-1} ] \boldsymbol{e}_1 - [p(S) - S^{-1}] \boldsymbol{v}_1 \| + \| V_k p( T_{k,k}^< ) \boldsymbol{e}_1 - p (S) \boldsymbol{v}_1 \| \\
	\leq & \| p ( T_{k, k}^< ) - ( T_{k, k}^< )^{-1} \| + \| p (S) - S^{-1} \|
	\end{align}
	with $V_k^\mathsf{H} V_k = \mathrm{I}$ and $\| \boldsymbol{e}_1 \| = \| \boldsymbol{v}_1 \| = 1$.
	Here, we used the shifted Lanczos decomposition~\eqref{eq:shiftedLanczos} and the identity 
	\begin{align}
	p (S) \boldsymbol{v}_1 = V_k p (T_{k, k}^<) \boldsymbol{e}_1,
	\end{align}
	cf.~\cite[Lemma~4.1]{MuscoMuscoSidford2018SODA}.
	Because of the interlacing property of eigenvalues~\cite[Theorem~4.3.17]{HornJohnson2013}
	\begin{align}
	\| p (S) - (S)^{-1} \| \leq \max_{t \in [ \lambda_1, \lambda_n ]} | p(z - t) - (z - t)^{-1} |
	\end{align}
	and 
	\begin{align}
	\| p ( T_{k, k}^< ) - ( T_{k, k}^< )^{-1} \| \leq \max_{t \in [ \lambda_1, \lambda_n ]} | p(z - t) - (z - t)^{-1} |.
	\end{align}
	Therefore, the assertion~\eqref{eq:errnrm} holds.
	\qed
\end{proof}

The shifted Lanczos method uses a short-term recurrence and may suffer from rounding errors.
Rounding errors in the Lanczos method may quickly lead to loss of orthogonality of the computed basis vectors $\boldsymbol{v}_i$ of the Krylov subspace and thus could cause a delay in convergence.
We leave an open problem of giving the bound in finite precision, cf.~\cite[Theorem~6.2]{MuscoMuscoSidford2018SODA}.

\subsection{Estimation of error} \label{sec:errest}
The bound presented in the previous section gives insights on the convergence of the shifted Lanczos method for quadratic forms but is not practically used as a stopping criterion in the iteration of Algorithm~\ref{alg:shiftedLanczos} because it needs to know the condition number of $z \mathrm{I} - A$.
To check if a satisfactory solution is obtained for each iteration in Algorithm~\ref{alg:shiftedLanczos} in practice, we formulate two estimates of the error~$\varepsilon_k$.

First, the shifted Lanczos decomposition~\eqref{eq:shiftedLanczos} gives
\begin{align}
S V_k = V_k T_k^< - \beta_k \boldsymbol{v}_{k+1} \boldsymbol{e}_k.
\end{align}
Multiplying this by $\boldsymbol{v}_1^\mathsf{T} S^{-1}$ from the left and $(T_k^<)^{-1} \boldsymbol{e}_1$ from the right gives
\begin{align}
\boldsymbol{e}_1 (T_k^<)^{-1} \boldsymbol{e}_1 = \boldsymbol{v}_1 S^{-1} \boldsymbol{v}_1 - \beta_k (\boldsymbol{v}_1 S^{-1} \boldsymbol{v}_{k+1}) [\boldsymbol{e}_k ( T_k^< )^{-1} \boldsymbol{e}_1 ].
\end{align}
Together with $\| \boldsymbol{v} \|^2$, we have
\begin{align}
\varepsilon_k = \beta_k \| \boldsymbol{v} \|^2 | ( \boldsymbol{v}_1^\mathsf{H} S^{-1} \boldsymbol{v}_{k+1} ) [ \boldsymbol{e}_k^\mathsf{T} (T_k^<)^{-1} \boldsymbol{e}_1 ] |.
\end{align}
With the representation~\eqref{eq:Sk_inv} and a positive integer~$d$, we approximate $\boldsymbol{v}_1^\mathsf{H} S^{-1} \boldsymbol{v}_{k+1}$ by
\begin{align}
\boldsymbol{v}_1^\mathsf{H} (S_{k+d})^{-1} \boldsymbol{v}_{k+1} = \boldsymbol{e}_1 ( T_{k+d}^< )^{-1} \boldsymbol{e}_{k+1}.
\end{align}
Therefore, we obtain the following estimate:
\begin{align}
\varepsilon_k \simeq \beta_k \| \boldsymbol{v} \|^2 | [ \boldsymbol{e}_1^\mathsf{T} (T_k^<)^{-1} \boldsymbol{e}_k ] [ \boldsymbol{e}_1 ( T_{k+d}^< )^{-1} \boldsymbol{e}_{k+1} ] | \equiv \mu_{k, d}.
\end{align}
This estimate~$\mu_{k, d}$ requires additional $d$ iterations and needs $\boldsymbol{e}_1^\mathsf{T} (T_k^<)^{-1} \boldsymbol{e}_k$ and $\boldsymbol{e}_1 ( T_{k+d}^< )^{-1} \boldsymbol{e}_{k+1}$ to compute.
The former is the $(1, k)$ entry of $( T_k^< )^{-1}$ and the latter is the $(1, k+1)$ entry of $( T_{k+d}^< )^{-1}$.
These quantities can be computed recursively~\cite[Sections~3.2, 3.4]{GolubMeurant2010} as follows:
\begin{align}
\boldsymbol{e}_1^\mathsf{T} ( T_k^< )^{-1} \boldsymbol{e}_k & = (-1)^{k-1} \frac{\beta_1 \beta_2 \cdots \beta_{k-1}}{\delta_1 \delta_2 \cdots \delta_k}, \\
\boldsymbol{e}_1^\mathsf{T} ( T_{k+d} )^{-1} \boldsymbol{e}_{k+1} & = (-1)^k \frac{\beta_1 \beta_2 \cdots \beta_k}{\delta_1 \delta_2 \cdots \delta_{k+d}} \varphi_2^{(k+d)} \varphi_3^{(k+d)} \cdots \varphi_d^{(k+d)},
\end{align}
where $\delta_k$ is defined in \eqref{eq:delta} and 
\begin{align}
\varphi_d^{(k+d)} = z - \alpha_{k+d}, \quad \varphi_j^{(k+d)} = z - \alpha_{k+j} - \frac{\beta_{k+j}^2}{\varphi_{j+1}^{(k+d)}}, \quad j = d-1, d-2, \dots, 2.
\end{align}
Therefore, we may update $\boldsymbol{e}_1^\mathsf{T} ( T_k^< )^{-1} \boldsymbol{e}_k$ by
\begin{align}
\boldsymbol{e}_1^\mathsf{T} ( T_{k+1}^< )^{-1} \boldsymbol{e}_{k+1} = - \frac{\beta_k}{\delta_{k+1}} \boldsymbol{e}_1^\mathsf{T} ( T_k^< )^{-1} \boldsymbol{e}_k,
\end{align}
To update $\boldsymbol{e}_1^\mathsf{T} ( T_{k+1}^< )^{-1} \boldsymbol{e}_{k+1}$ for each shift per iteration, one multiplication by $\beta_{k-1}$ and one division by $\delta_k$ are required.
To compute $\varphi_j^{(k+d)}$ for $j = 2$, $3$, $\dots$, $d$ for each shift per iteration, we require one minus in $z - \alpha_{k+d}$, $d-2$ minuses and $d-2$ divisions in $(z-\alpha_{k+j}) - \beta_{k+j}^2 / \varphi_{j+1}^{(k+d)}$ if one stores $z - \alpha_{k+j}$ and $\beta_{k+j}^2$ for $j = 2$, $3$, $\dots$, $d-1$ computed in Line~7 of Algorithm~\ref{alg:shiftedLanczos}.
To update $\boldsymbol{e}_1^\mathsf{T} ( T_{k+d} )^{-1} \boldsymbol{e}_{k+1}$ for each shift per iteration, $d$ multiplications by $\varphi_j^{(k+d)}$, one multiplication by $\beta_k$, and one division by $\delta_{k+d}$ are required.

Second, the triangular equation gives
\begin{align}
\varepsilon_k & = | L_k - L_{k+d} + L_{k+d} - \boldsymbol{v}^\mathsf{H} S^{-1} \boldsymbol{v} | \\
& \leq | L_k - L_{k+d} | + | L_{k+d} - \boldsymbol{v}^\mathsf{H} S^{-1} \boldsymbol{v} |.
\end{align}
Under the assumption that the approximation error for the ($k+d$)th iteration is significantly smaller than that for the $k$th iteration, we obtain an estimate for the $k$th iteration
\begin{align}
\varepsilon_k \simeq | L_k - L_{k+d} | \equiv \nu_{k, d}.
\end{align}
This estimate~$\nu_{k, d}$ requires additional $d$ iterations and one minus in $L_k - L_{k+d}$.
An analogous estimate can be found in \cite[Section~4]{StrakosTichy2002ETNA} for the $A$-norm error in the CG method under a similar assumption.
The shifted coefficient matrix~$z \mathrm{I} - A$ may not be Hermitian and may not form the $(z \mathrm{I} - A)$-norm in the usual sense.

\section{Numerical experiments} \label{sec:exp}
Numerical experiments were performed to compare the proposed method with the previous methods---shifted COCG method (Algorithm~\ref{alg:shiftedCOCG}), shifted COCR method (Algorithm~\ref{alg:shiftedCOCR}), shifted MINRES method (Algorithm~\ref{alg:shiftedMINRES}), shifted Lanczos method (Algorithm~\ref{alg:shiftedLanczos}, and direct solver using the MATLAB function \texttt{mldivide} for solving~\eqref{eq:linsys}---in terms of the number of iterations and CPU time.

All computations were performed on a computer with an Intel Xeon E5-2670 v2 2.50 GHz CPU, 256 GB of random-access memory (RAM), and CentOS 6.10.
All programs were coded and run in MATLAB R2019a in double-precision floating-point arithmetic with unit round-off at~$2^{-52} \simeq 2.2 \cdot 10^{-16}$.

Table~\ref{tb:matrix_info} gives information about the test matrices, including the size of each matrix, density of nonzero entries [\%], (estimated) condition number, and application from which the matrix arose.
The condition number was estimated using the MATLAB function~\texttt{condest}.
Matrices mhd1280b, apache2, CurlCurl\_3 and thermal2 are from~\cite{DavisHu2011}; the other matrix, VCNT1000000std, is from~\cite{HoshiImachiKuwataKakudaFujitaMatsui2019JJIAM}.
Matrix mhd1280b is complex symmetric and matrix conf5.4-00l8x8-2000 is complex Hermitian and indefinite, whereas the others are real symmetric.
The Cholesky factorization~\cite{Rump2006BIT} proved that matrices mhd1280b and apache2 are positive definite.
According to the information in \cite{DavisHu2011}, matrix CurlCurl\_3 is not positive definite and matrix thermal2 is positive definite.

The condition number for CurlCurl\_3 could not be computed because of insufficient computer memory.
Vector $\boldsymbol{v}$ was simply set to $\boldsymbol{v} = n^{-1/2} \boldsymbol{e}$ for reproducibility, where $\boldsymbol{e}$ is the all-ones vector.
The shifts were set to
\begin{equation}
z_i = \exp \left(- \frac{2i+1}{2m} \pi \mathrm{i} \right), \quad i = 1, 2, \dots, m, \quad m = 16,
\label{eq:z_i}
\end{equation}
where $\mathrm{i}$ is the imaginary unit.
Because $\mathrm{Im} (z_i) \ne 0$, the conditions in Theorem~\ref{th:breakdown_free} are satisfied and the shifted Lanczos method is breakdown-free.
These shifts demonstrate typical choices of quadrature points in the projection method for eigenproblems~\cite{SakuraiTadano2007}.

\setlength{\tabcolsep}{6pt}
\begin{table}[t]
	\scriptsize
	\centering
	\caption{Information on test matrices}
	\begin{tabular}{lrrrl}
		\hline\noalign{\smallskip}
		Matrix & $n$  & Density [\%] & \texttt{condest} & Application \\
		\noalign{\smallskip}\hline\noalign{\smallskip}
		mhd1280b & 1{,}280 & $1.4 \cdot 10^{0\phantom{-}}$ & $6.0 \cdot 10^{12}$ & Magnetohydrodynamics \\
		conf5.4-00l8x8-2000 & 49{,}152 & $8.2 \cdot 10^{-2}$ & $3.6\cdot10^{4\phantom{0}}$ & Quantum chromodynamics \\
		apache2 &  715{,}176 & $9.4\cdot10^{-4}$ & $5.3\cdot10^{6\phantom{0}}$ & Structural analysis \\
		VCNT1000000std & 1{,}000{,}000 & $4.0\cdot10^{-3}$ & $1.2\cdot10^{7\phantom{0}}$ & Quantum mechanics \\
		CurlCurl\_3 & 1{,}219{,}574 & $9.1\cdot10^{-4}$ & --- & Electromagnetic analysis \\
		thermal2 & 1{,}228{,}045 & $5.7\cdot10^{-4}$ & $7.5\cdot10^{6\phantom{0}}$ & Steady-state thermal analysis \\
		\noalign{\smallskip}\hline
	\end{tabular}
	\label{tb:matrix_info}
	\begin{minipage}{0.98\hsize}
		Matrix: name of the matrix, $n$: size of the matrix, density: density of nonzero entries, \texttt{condest}: condition number estimated by using the MATLAB function \texttt{condest}, application: application from which the matrix arises.
	\end{minipage}
\end{table}

\subsection{Comparisons in terms of CPU time} \label{sec:practical}
We compare the methods in terms of CPU time.
Figure~\ref{fig:practical} shows the relative error versus the number of iterations for the compared methods on test matrices.
The shifted methods terminated when the largest relative error among $i = 1$, $2$, $\dots$, $16$ became less than or equal to $10^{-10}$.
Here, we adopted the numerical solution computed by MATLAB's sparse direct solver as the exact solution.
Table~\ref{tb:direct_solver} gives the largest relative residual norm $\max_i \| \boldsymbol{v} - (z_i \mathrm{I} - A)	 \tilde{\boldsymbol{x}}^{(i)} \| / \| \boldsymbol{b} \|$ for MATLAB's sparse direct solver for the test matrices, where $\tilde{\boldsymbol{x}}^{(i)}$ is the numerical solution of the linear system~$(z_i \mathrm{I}-A) \boldsymbol{x}^{(i)} = \boldsymbol{b}$ computed by the sparse direct solver.
The sparse direct solver was accurate on mhd1280b, VCNT1000000std, conf5.4-00l8x8-2000, and thermal2 and not accurate on apache2 and CurlCurl\_3.
The convergence curve plotted corresponds to the case where the required number of iterations is the largest among shifts $z_i$, $i=1$, $2$, $\dots$, $16$, because this case determines the total CPU time.
Such cases are for $z_1$ on mhd1280b, conf5.4-00l8x8-2000, apache2, and thermal2; for $z_6$ on CurlCurl\_3; and for $z_16$ on VCNT1000000std.
The shifted Lanczos method converged faster than the other methods on apache2, CurlCurl\_3, and thermal2 and was competitive with the shifted MINRES method on VCNT1000000std.
The shifted COCR and MINRES methods were almost identical on CurlCurl\_3.
The convergence curves of the four methods on VCNT1000000std and CurlCurl\_3 decreased monotonically.
Although the shifted MINRES method is monotonically non-increasing in terms of the residual norm, Figures~\ref{fig:practical}~\subref{fig:apache2} and \ref{fig:practical}~\subref{fig:thermal2} show that the shifted MINRES method does not necessarily converge monotonically in terms of the error.
The convergence curves of the four methods are oscillatory on apache2 and thermal2. 
The shifted Lanczos method did not reach the stopping criterion on CurlCurl\_3, whereas the shifted COCR and MINRES methods did.
The curves of the shifted MINRES and Lanczos methods for conf5.4-00l8x8-2000, those of the shifted COCG and Lanczos methods for VCNT1000000std, those of the shifted COCG and Lanczos methods for CurlCurl\_3, and those of the shifted COCR and MINRES methods for CurlCurl\_3 seem to overlap.

\begin{figure}[t!]
	\centering
	\begin{minipage}{0.49\hsize}
		\centering\includegraphics[scale=0.6]{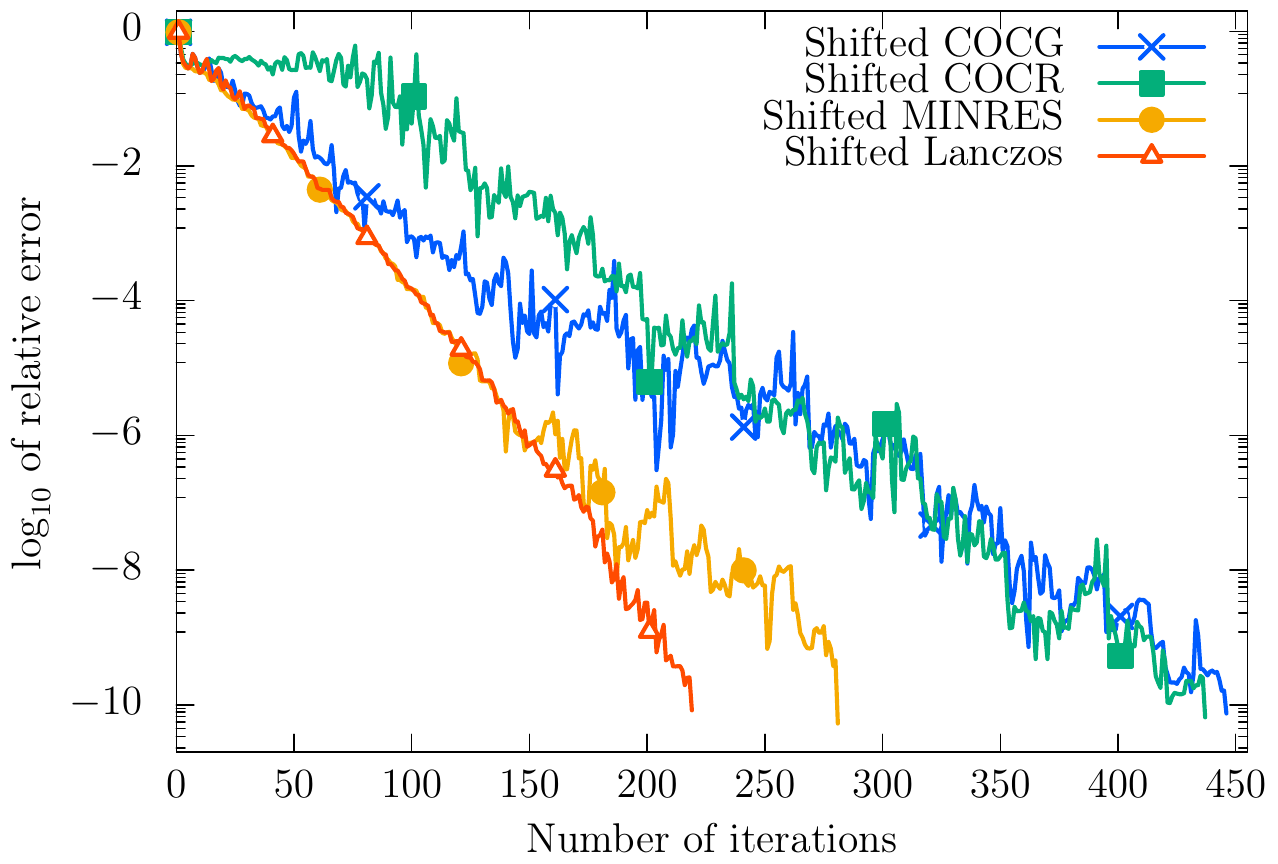}
		\subcaption{mhd1280b.}
		\label{fig:mhd1280b}
	\end{minipage}
	\begin{minipage}{0.49\hsize}
		\centering\includegraphics[scale=0.6]{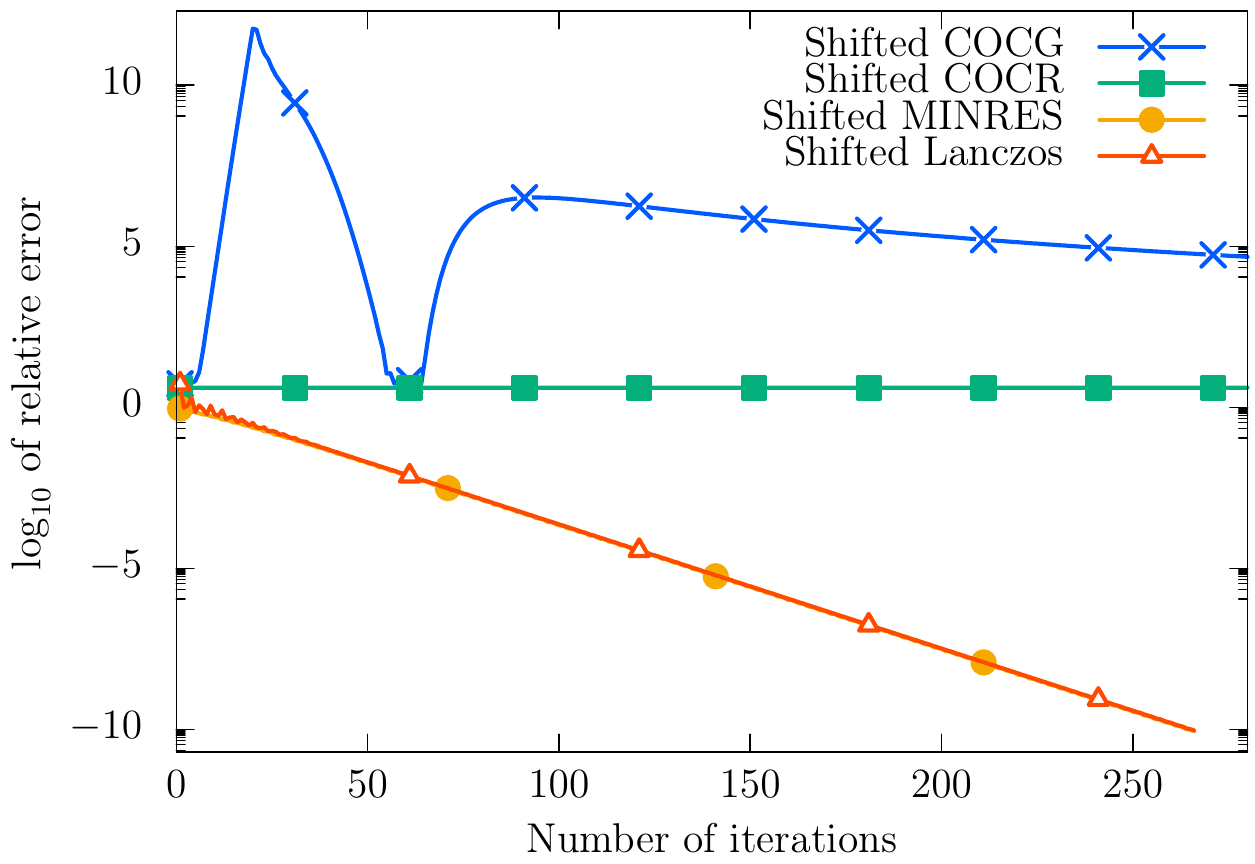}
		\subcaption{conf5.4-00l8x8-2000.}
		\label{fig:conf5.4-00l8x8-2000}
	\end{minipage}
	\begin{minipage}{0.49\hsize}
		\centering\includegraphics[scale=0.6]{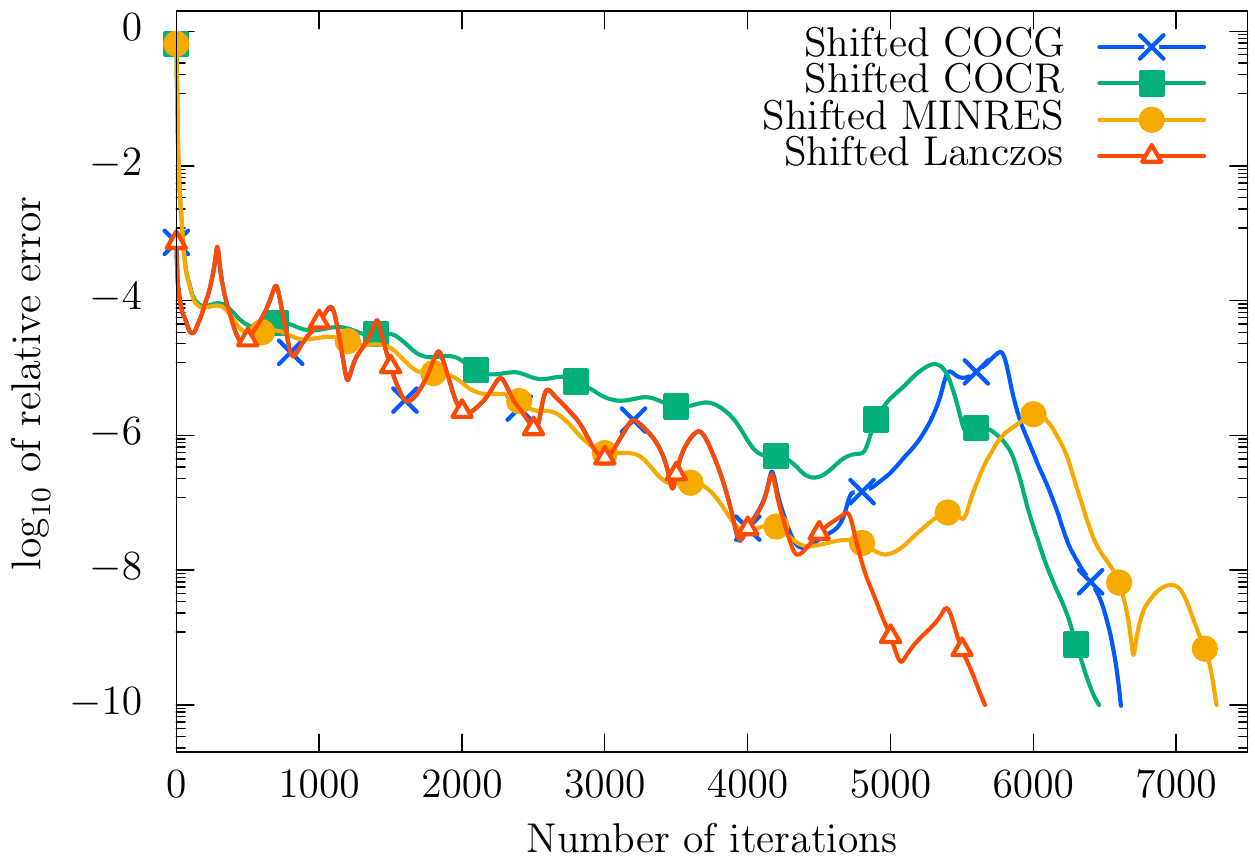}
		\subcaption{apache2.}
		\label{fig:apache2}
	\end{minipage}
	\begin{minipage}{0.49\hsize}
		\centering\includegraphics[scale=0.6]{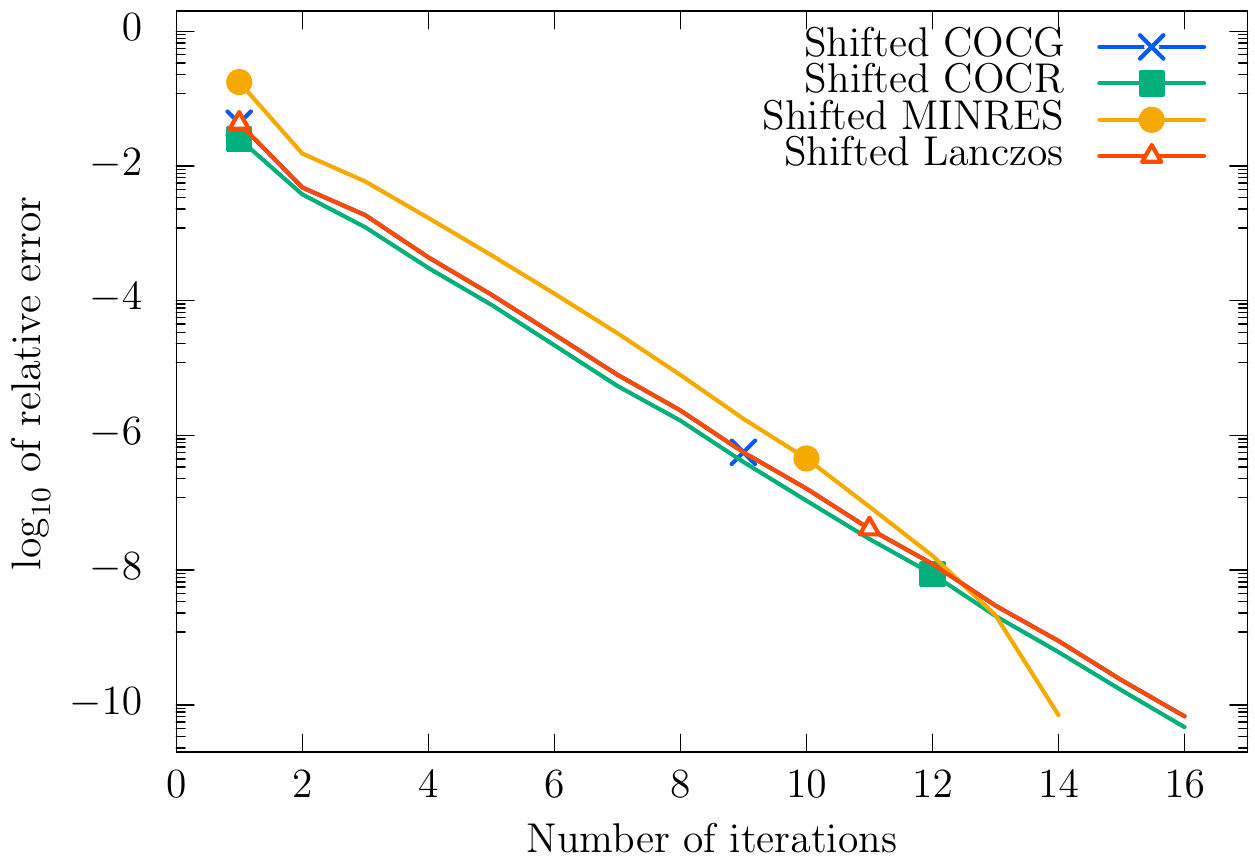}	
		\subcaption{VCNT1000000std.}
		\label{fig:VCNT1000000std}
	\end{minipage}
	\begin{minipage}{0.49\hsize}
		\centering\includegraphics[scale=0.58]{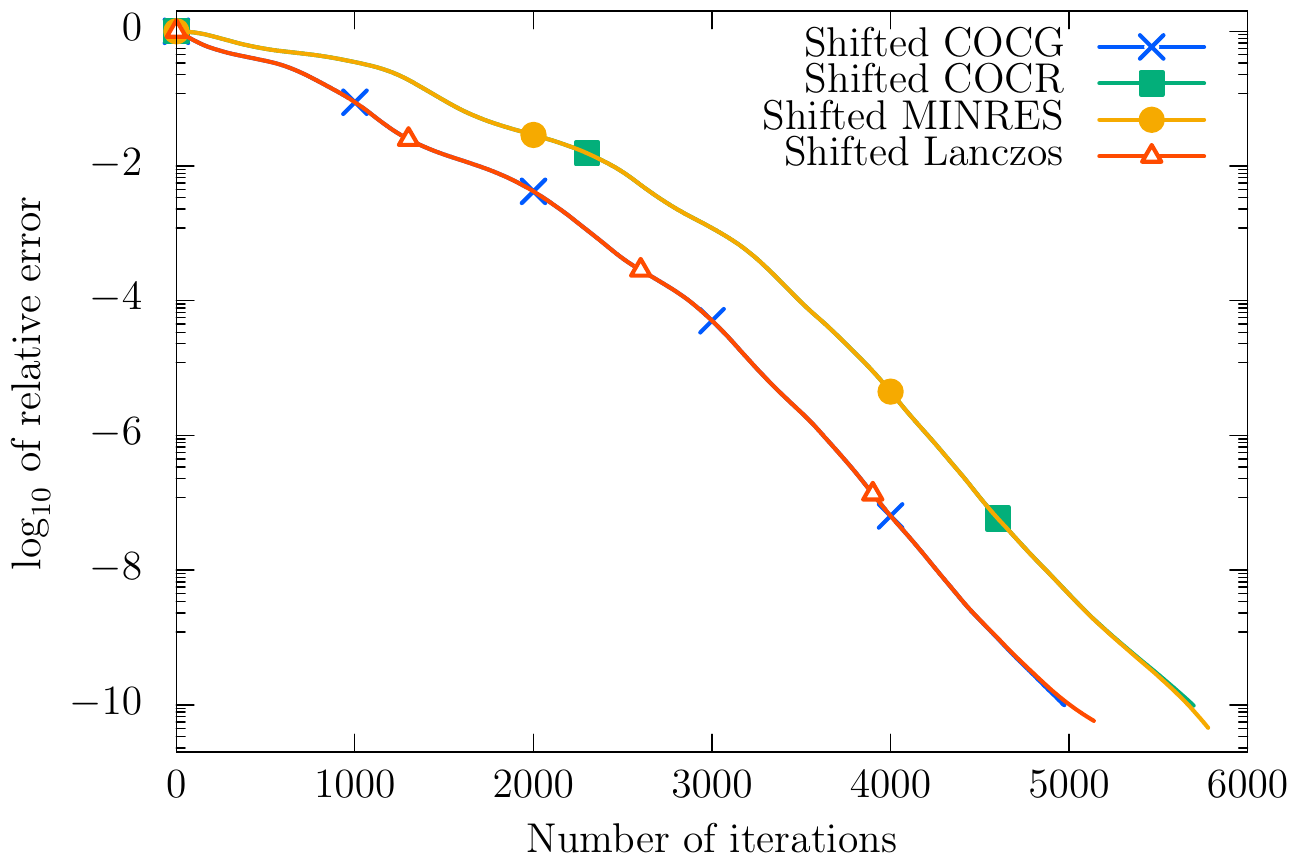}
		\subcaption{CurlCurl\_3.}
		\label{fig:CurlCurl_3}
	\end{minipage}
	\begin{minipage}{0.49\hsize}
		\centering\includegraphics[scale=0.58]{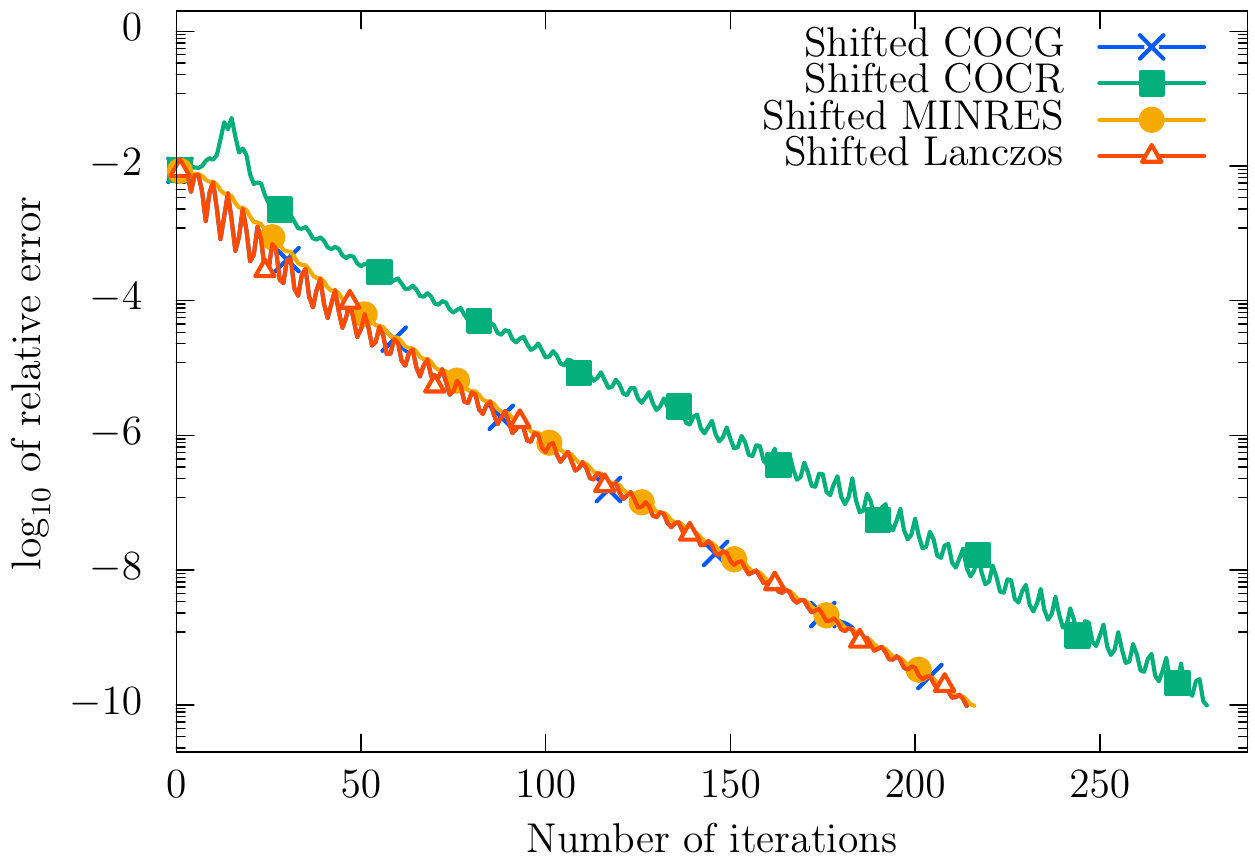}
		\subcaption{thermal2.}
		\label{fig:thermal2}
	\end{minipage}	
	\caption{Relative error vs.\ number of iterations.}
	\label{fig:practical}
\end{figure}

Table~\ref{tb:time} gives the CPU time in seconds taken by the four methods.
The symbol~* stands for the least CPU time for each matrix.
The shifted Lanczos method was competitive with or faster than the MINRES method except on CurlCurl\_3, whereas the shifted COCR method took more CPU time.
CCOG required fewer iterations than the shifted Lanczos method but took more CPU time on CurlCurl\_3 and thermal2.
This is because the shifted COCG (and COCR) method used $z_1 \ne \mathbb{R}$ as the seed, and its complex symmetric Lanczos iterations in Line~5 of Algorithm~\ref{alg:shiftedCOCG} were performed in complex arithmetic.
When setting the seed to zero, its complex symmetric Lanczos iterations were performed with real operations; however, the method required more CPU time and iterations than the case with shift $z_1$.

\setlength{\tabcolsep}{7pt}
\begin{table}[t]
	\scriptsize
	\centering
	\caption{Maximum relative residual norm for the direct method}
	\begin{tabular}{lrrrl}
		\hline\noalign{\smallskip}
		Matrix & $\max_i \| \boldsymbol{v} - (z_i \mathrm{I} - A) \tilde{\boldsymbol{x}}^{(i)} \| / \| \boldsymbol{v} \|$ \\
		\noalign{\smallskip}\hline\noalign{\smallskip}
		mhd1280b & $2.0\cdot10^{-16}$ \\
		conf5.4-00l8x8-2000 & $1.0\cdot10^{-15}$ \\
		apache2 & $4.6\cdot10^{-12}$ \\
		VCNT1000000std & $9.4\cdot10^{-16}$ \\
		CurlCurl\_3 & $5.0\cdot10^{-12}$\\
		thermal2 & $6.4\cdot10^{-16}$ \\
		\noalign{\smallskip}\hline
	\end{tabular}
	\label{tb:direct_solver}
	\begin{minipage}{0.98\hsize}
		Matrix: name of the matrix, $\max_i \| \boldsymbol{v} - (z_i \mathrm{I} - A) \tilde{\boldsymbol{x}}^{(i)} \| / \| \boldsymbol{v} \|$: maximum relative error norm.
	\end{minipage}
\end{table}

\begin{table}[t]
	\scriptsize
	\centering
	\caption{Number of iterations and CPU times [s] on test matrices for the methods compared}
	\begin{tabular}{lrrrrrr}
		\hline\noalign{\smallskip}
		Method & \multicolumn{2}{c}{mhd1280b} & \multicolumn{2}{c}{conf5.4-00l8x8-2000} &
		\multicolumn{2}{c}{apache2} \\
		& iter & time & iter & time & iter & time 
		\\
		\noalign{\smallskip}\hline\noalign{\smallskip}						
		\texttt{mldivide} & & 0.08 & & 6{,}505 
		& & 1{,}978 
		\\
		Shifted COCG & 446 & 0.05 & --- &
		& 6{,}613 & 175	\\
		Shifted COCR & 437 & 0.04 & --- &
		& 6{,}461 & 163 \\			
		Shifted MINRES & 281 & 0.03 & 266 & 1.67 
		& 7{,}283 & 95.4 \\			
		Shifted Lanczos & 219 & *0.02 & 266 & *1.52 
		& 5{,}662 & *68.4 \\
		\noalign{\smallskip}\hline	
	\end{tabular}
	\begin{tabular}{lrrrrrr}
		\hline\noalign{\smallskip}
		Method & \multicolumn{2}{c}{VCNT1000000std} & \multicolumn{2}{c}{CurlCurl\_3} & \multicolumn{2}{c}{thermal2} \\
		& iter & time & iter & time & iter & time \\
		\noalign{\smallskip}\hline\noalign{\smallskip}			
		\texttt{mldivide} & & 555 & & 749{,}696 & & 563 \\
		Shifted COCG & 16 & 1.18 & 4{,}974 & 246 & 214 & 11.40 \\
		Shifted COCR & 16 & 1.16 & 5{,}698 & 382 & 279 & 14.57 \\
		Shifted MINRES & 14 & *0.73 & 5{,}779 & 156 & 216 & 6.53 \\
		Shifted Lanczos & 16 & 0.81 & 5{,}139 & *135 & 214 & *6.17 \\
		\noalign{\smallskip}\hline	
	\end{tabular}
	\label{tb:time}
\end{table}

\subsection{Spectrum and shift}
We illustrate the effect of the spectrum of $A$ and shift~$z$ on the convergence.
The condition number of shifted matrix~$z \mathrm{I} - A$ is given by 
\begin{align}
\kappa = \frac{\max_i | z - \lambda_i |}{\min_i | z - \lambda_i |}.
\end{align}
This means that the condition number has a connection between the spectrum and shift.
For convenience of the computation of a spectrum, we take a small matrix mhd1280b as an example.
The largest and smallest eigenvalues of the matrix are approximately $70.32$ and $1.48 \cdot 10^{-11}$, respectively.
Substitute the shift $z = \lambda_1 + \zeta \mathrm{i}$ for $\zeta = 10^j$, $j = -1$, $-2$, $-3$, and $-4$ for controlling the condition number.
Then, the corresponding condition numbers of $z \mathrm{I} - A$ are as given in Table~\ref{tb:spectrum_shift}.
Table~\ref{tb:spectrum_shift} gives the required number of iterations and CPU time for the compared methods for each value of $\zeta$ value.
The stopping criterion was that the relative error became less than or equal to $10^{-10}$.
The table also shows that the shifted Lanczos method was competitive with or outperformed other methods in terms of the CPU time.
It shows that as the condition number increases, the required number of iterations tends to increase.

\begin{table}[t]
	\scriptsize
	\centering
	\caption{Number of iterations and CPU times [s] on test matrices for the methods compared for different values of shift~$z = \lambda_1 + \zeta \mathrm{i}$}
	\begin{tabular}{lrrrrrrrr}
		\hline\noalign{\smallskip}
		$\zeta$ & \multicolumn{2}{c}{$10^{-1}$} & \multicolumn{2}{c}{$10^{-2}$} & \multicolumn{2}{c}{$10^{-3}$} & \multicolumn{2}{c}{$10^{-4}$} \\
		condition number & \multicolumn{2}{c}{$1.1\cdot10^3$} & \multicolumn{2}{c}{$1.2\cdot10^4$} & \multicolumn{2}{c}{$1.0\cdot 10^5$} & \multicolumn{2}{c}{$8.9\cdot 10^5$} \\
		& iter & time & iter & time & iter & time & iter & time \\
		\noalign{\smallskip}\hline\noalign{\smallskip}
		Shifted COCG    & 118 & *0.01 & 357 &  0.03 & 1{,}087 &  0.10 & 3{,}013 &  0.27 \\
		Shifted COCR    & 117 & *0.01 & 331 &  0.03 & 1{,}018 &  0.10 & 3{,}129 &  0.29 \\
		Shifted MINRES  &  84 & *0.01 & 231 &  0.03 &     690 &  0.07 & 1{,}988 &  0.19 \\
		Shifted Lanczos &  76 & *0.01 & 226 & *0.02 &     680 & *0.06 & 1{,}894 & *0.17 \\
		\noalign{\smallskip}\hline
	\end{tabular}
	\label{tb:spectrum_shift}
\end{table}
\begin{figure}[t]
	\centering
	\begin{minipage}{0.49\hsize}
		\centering\includegraphics[scale=0.58]{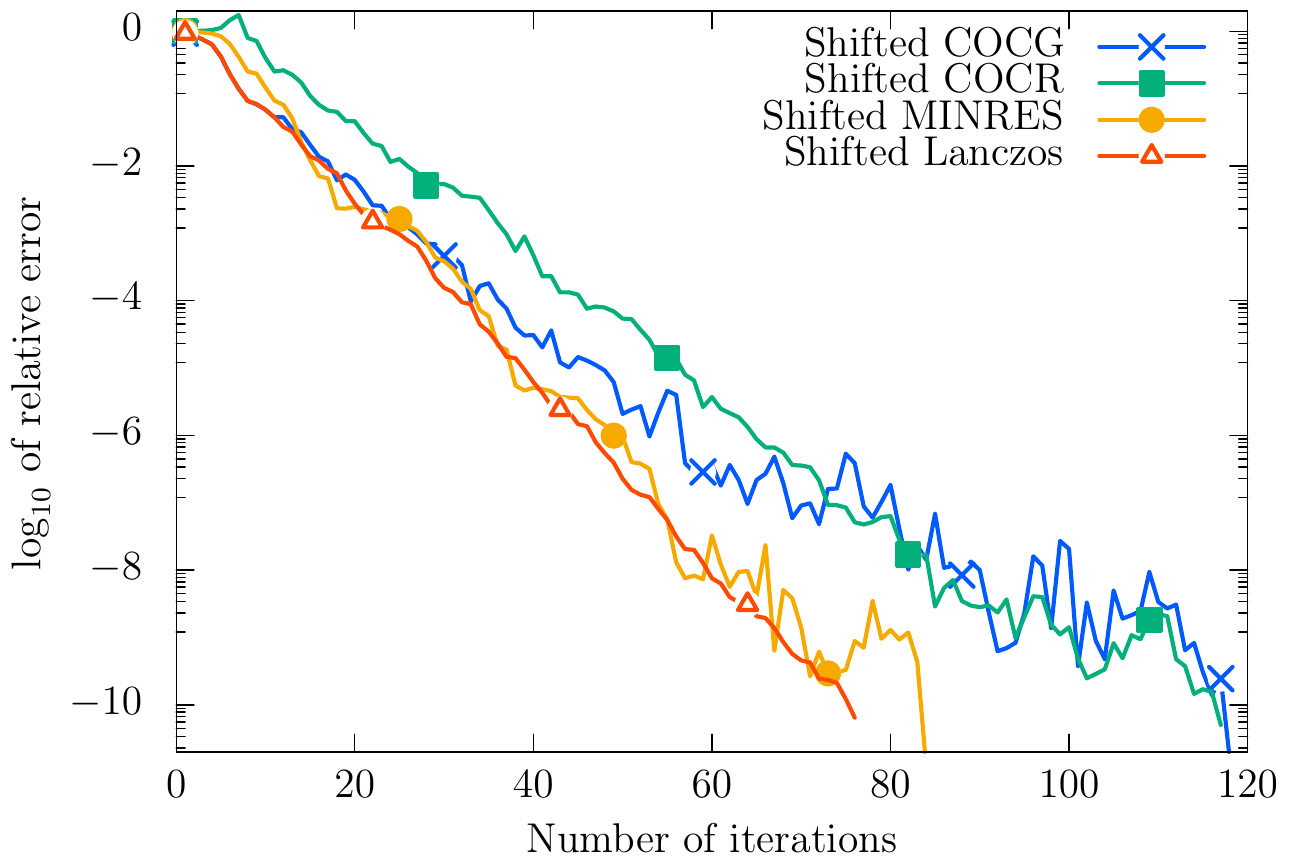}
		\subcaption{$\zeta = 10^{-1}$.}
		\label{fig:zeta=1e-1}
	\end{minipage}
	\begin{minipage}{0.49\hsize}
		\centering\includegraphics[scale=0.58]{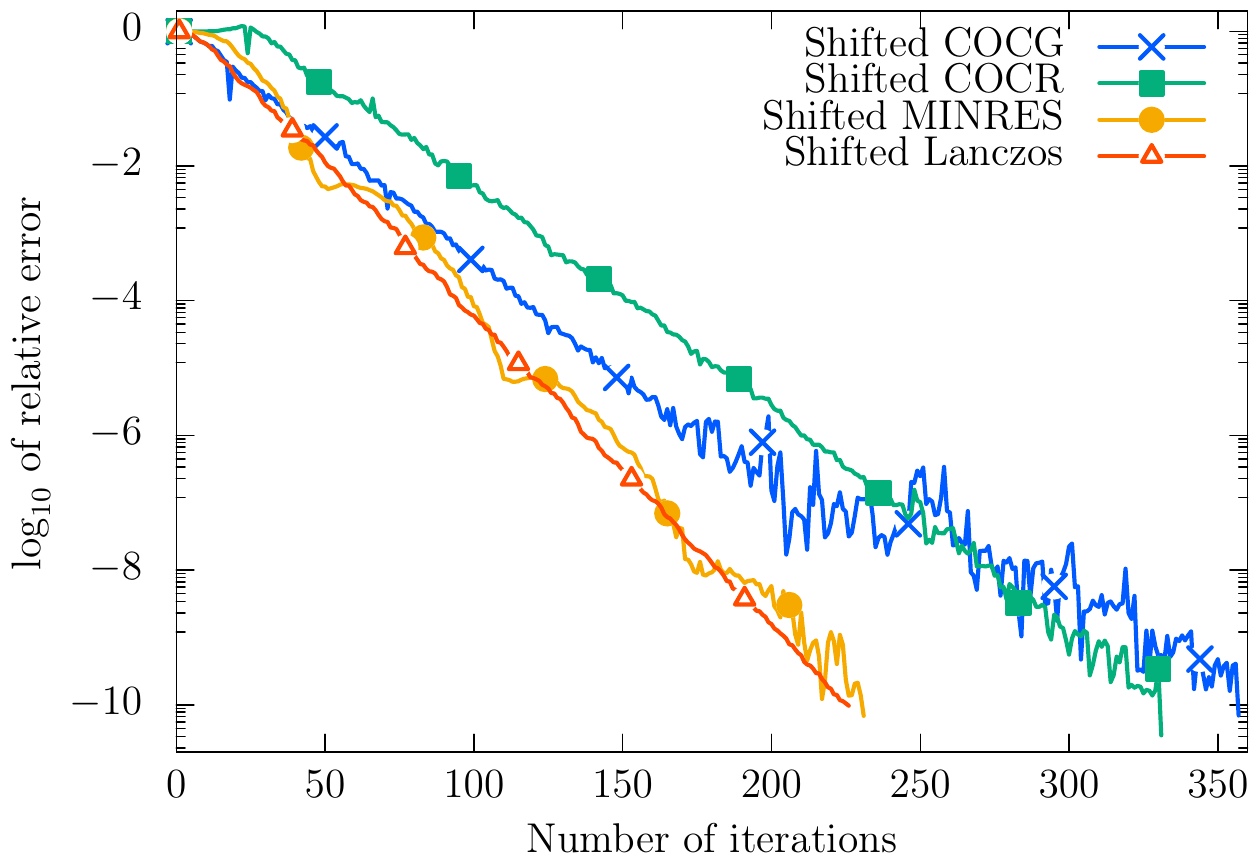}
		\subcaption{$\zeta = 10^{-2}$.}
		\label{fig:zeta=1e-2}
	\end{minipage}
	\begin{minipage}{0.49\hsize}
		\centering\includegraphics[scale=0.58]{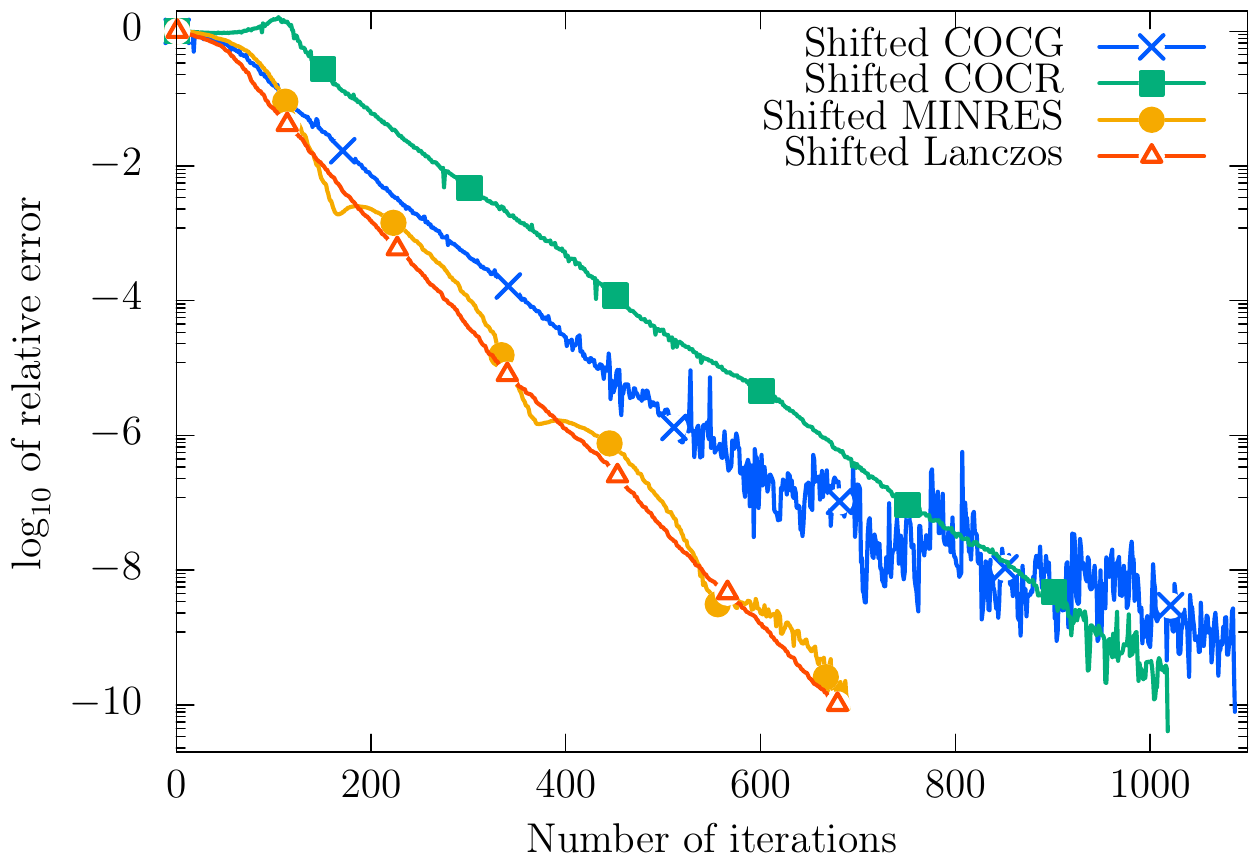}
		\subcaption{$\zeta = 10^{-3}$.}
		\label{fig:zeta=1e-3}
	\end{minipage}
	\begin{minipage}{0.49\hsize}
		\centering\includegraphics[scale=0.58]{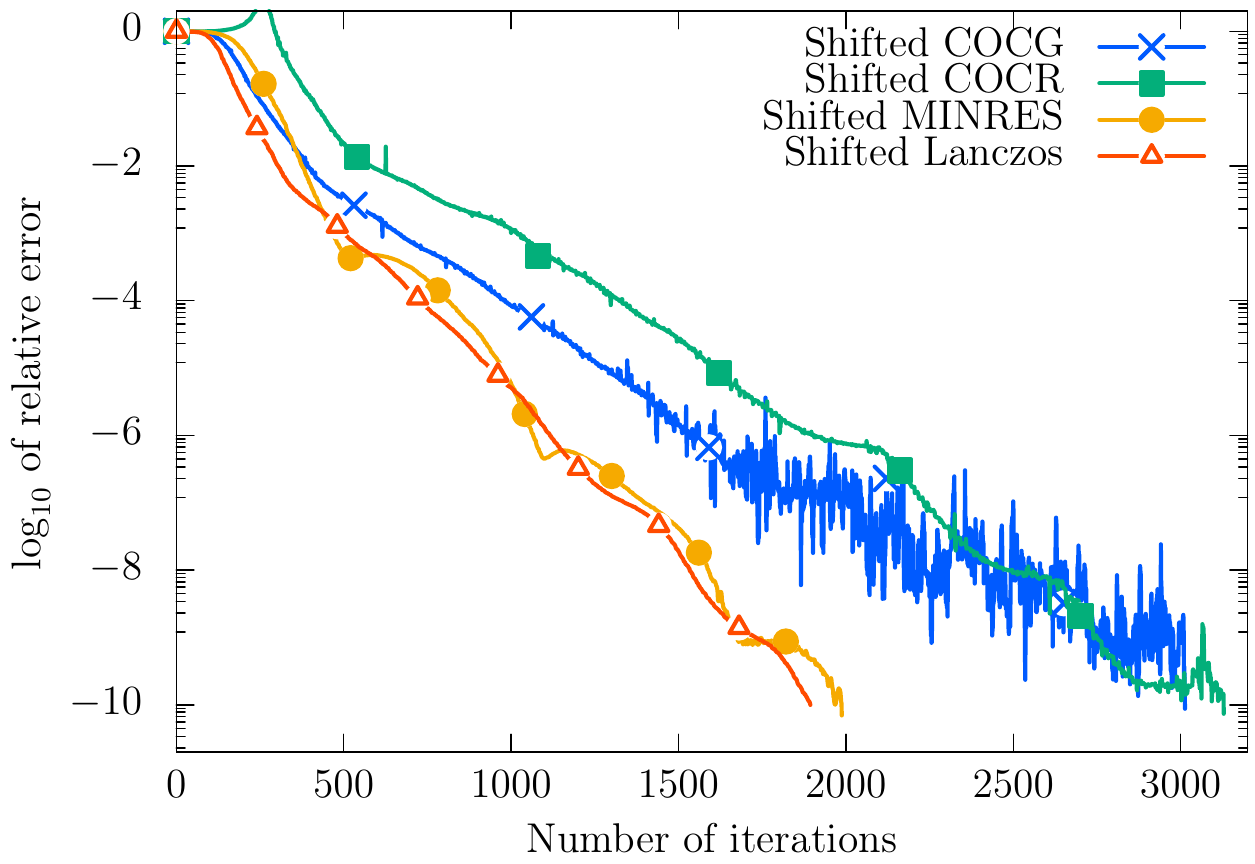}	
		\subcaption{$\zeta = 10^{-4}$.}
		\label{fig:zeta=1e-4}
	\end{minipage}
	\caption{Relative error vs.\ number of iterations for mdh1208b for different values of $\zeta$.}
	\label{fig:spectrum_shift}
\end{figure}

\subsection{Estimation of error}
We tested the estimates~$\mu_{k, d}$ and $\nu_{k, d}$ developed in Section~\ref{sec:errest} on the matrices used in Section~\ref{sec:practical}.
Figure~\ref{fig:practical_e} shows the relative error and its estimates~$\mu_{k, d}$ and $\nu_{k, d}$ with $d = 5$ versus the number of iterations.
Here, the plotted curve corresponds to the case where the required number of iterations was the largest among shifts $z_i$, $i=1$, $2$, $\dots$, $16$.
For presentation, these estimates are normalized by $\boldsymbol{v}^\mathsf{H} \tilde{\boldsymbol{x}}^{(i)}$.
Both estimates were accurate on conf5.4-00l8x8-2000, VCNT1000000std, and thermal2.
The estimate~$\zeta_{k, 5}$ was accurate on mdh1280b.
Both estimates underestimated the error on apache2 and CurlCurl\_3 because the assumption made in Section~\ref{sec:errest} might not hold and the shifted Lanczos method lacks the monotonicity of an error norm such as the one in the CG method~(cf.~\cite{StrakosTichy2002ETNA}).
The oscillations of the estimates are similar to those observed for the errors on apache2 and thermal2.
However, we have no theoretical justification of these estimates at this time and when they capture the convergence well, in particular for finite-precision arithmetic, is open.

\begin{figure}[t!]
	\centering
	\begin{minipage}{0.49\hsize}
		\centering\includegraphics[scale=0.6]{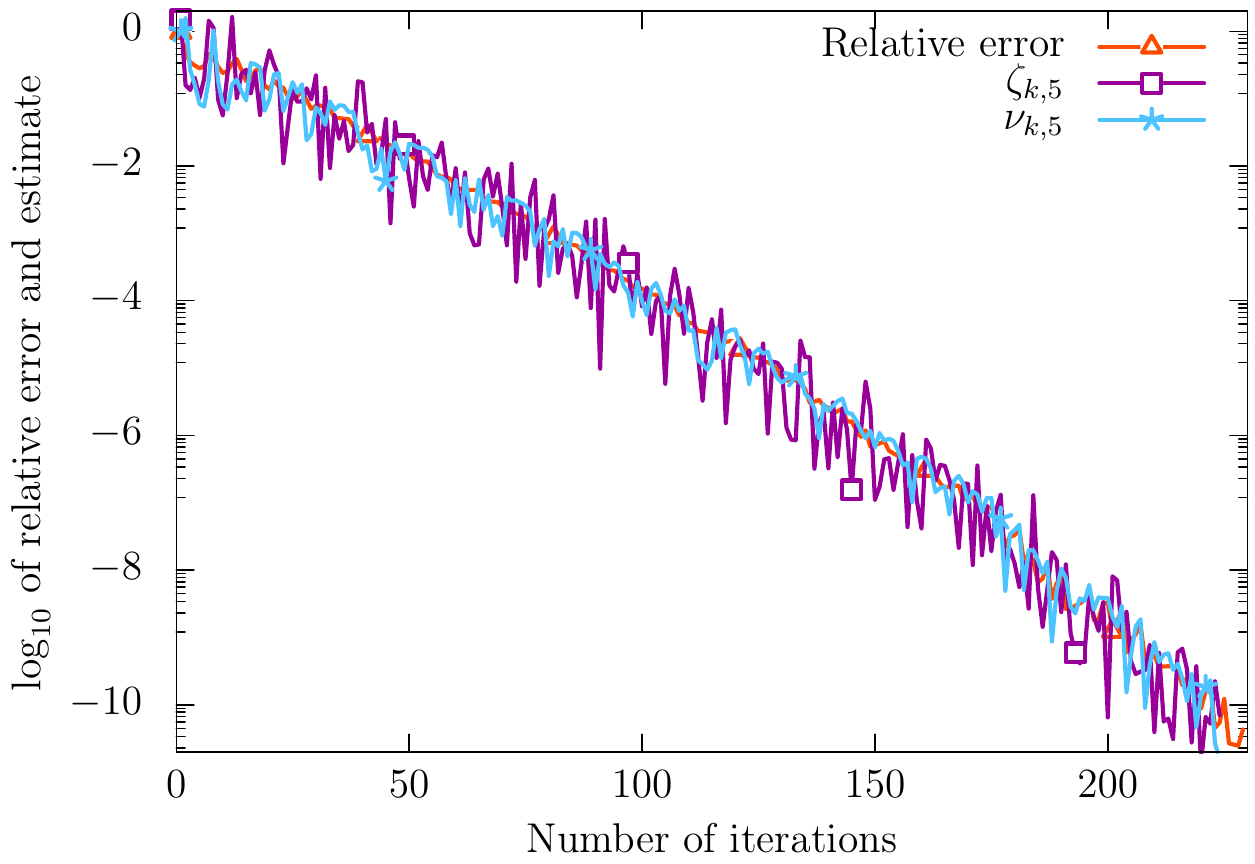}
		\subcaption{mhd1280b.}
		\label{fig:mhd1280b_e}
	\end{minipage}
	\begin{minipage}{0.49\hsize}
		\centering\includegraphics[scale=0.6]{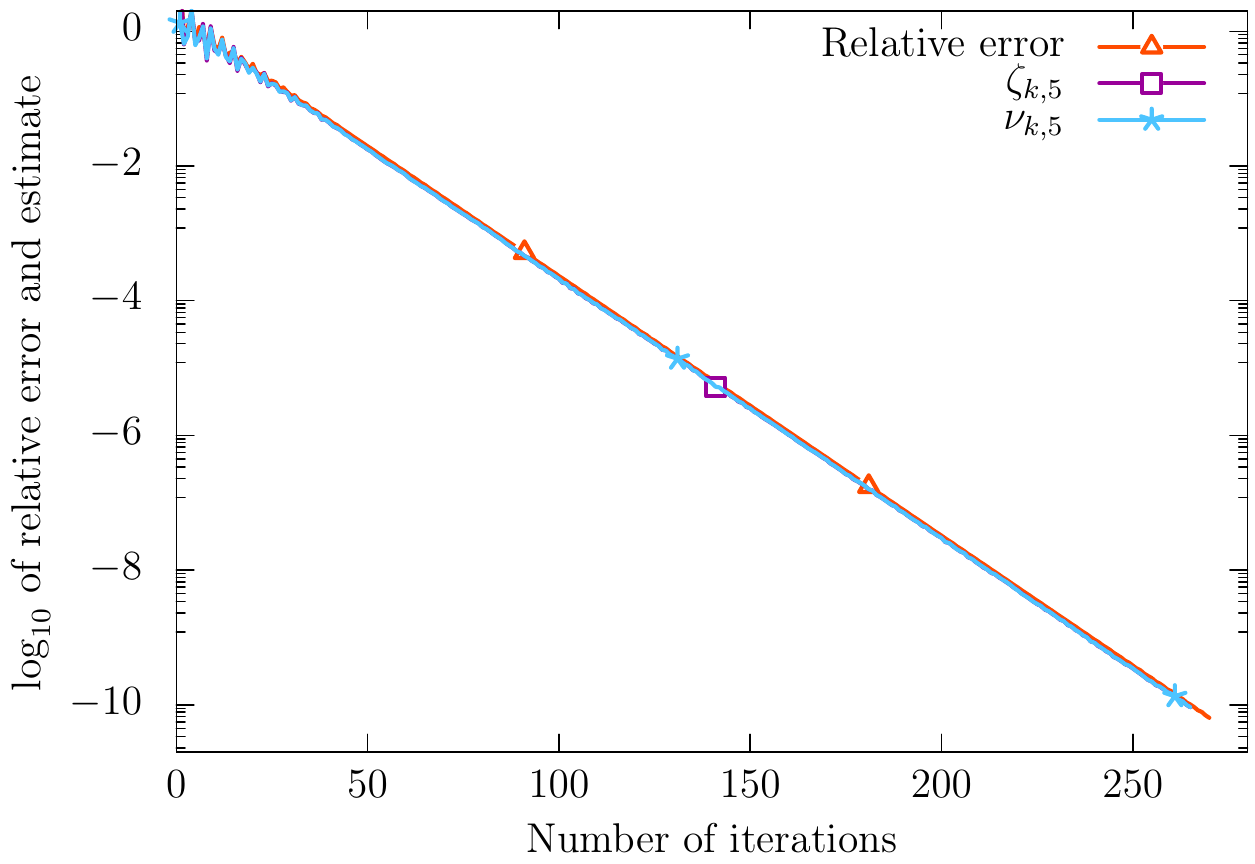}
		\subcaption{conf5.4-00l8x8-2000.}
		\label{fig:conf5.4-00l8x8-2000_e}
	\end{minipage}
	\begin{minipage}{0.49\hsize}
		\centering\includegraphics[scale=0.6]{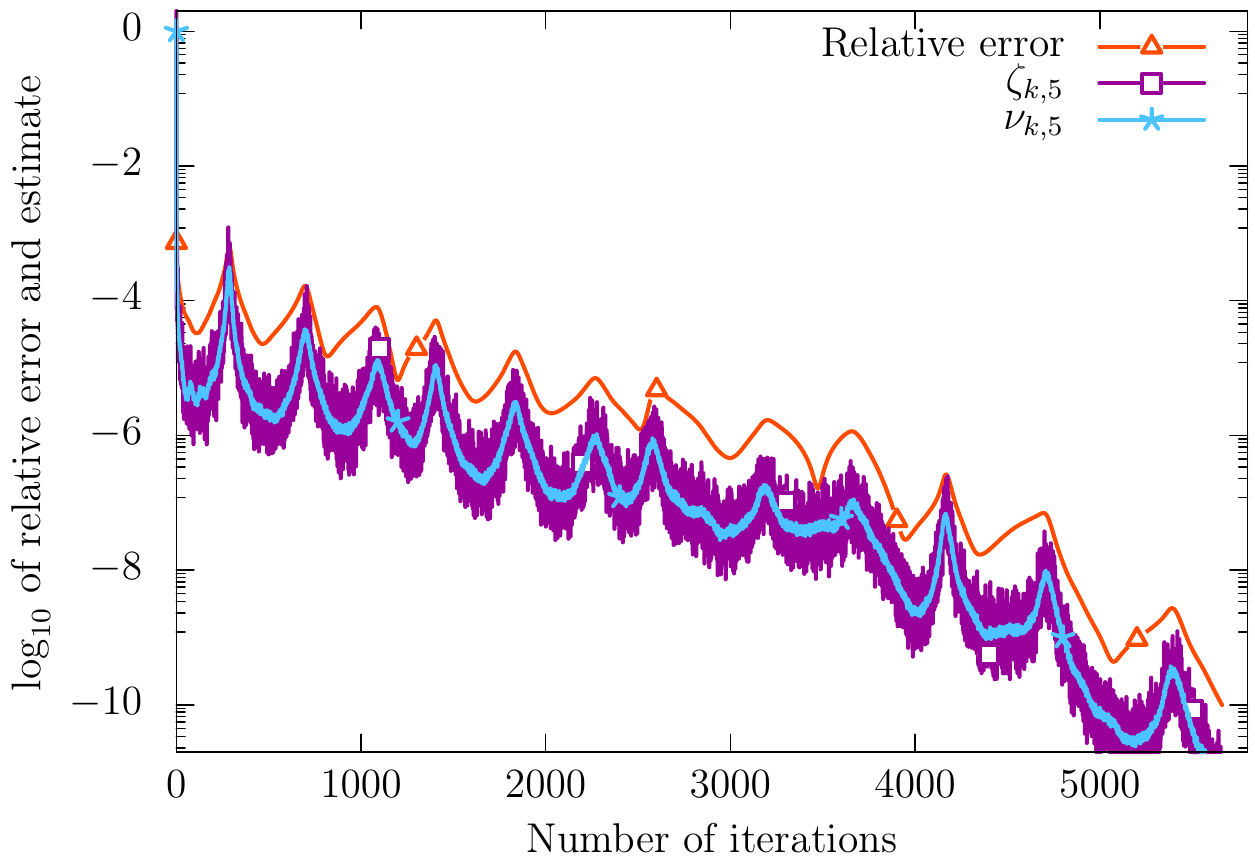}
		\subcaption{apache2.}
		\label{fig:apache2_e}
	\end{minipage}
	\begin{minipage}{0.49\hsize}
		\centering\includegraphics[scale=0.6]{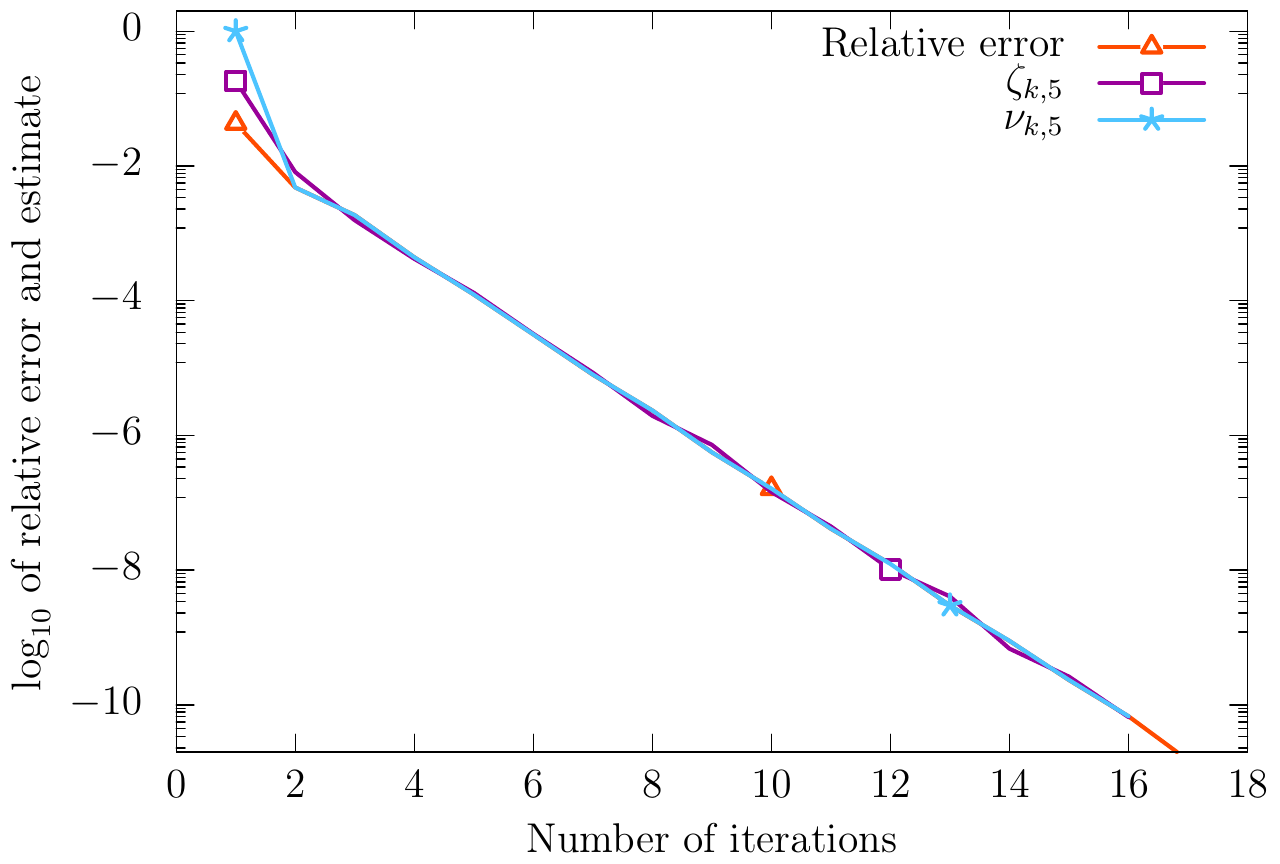}	
		\subcaption{VCNT1000000std.}
		\label{fig:VCNT1000000std_e}
	\end{minipage}	
	\begin{minipage}{0.49\hsize}
		\centering\includegraphics[scale=0.6]{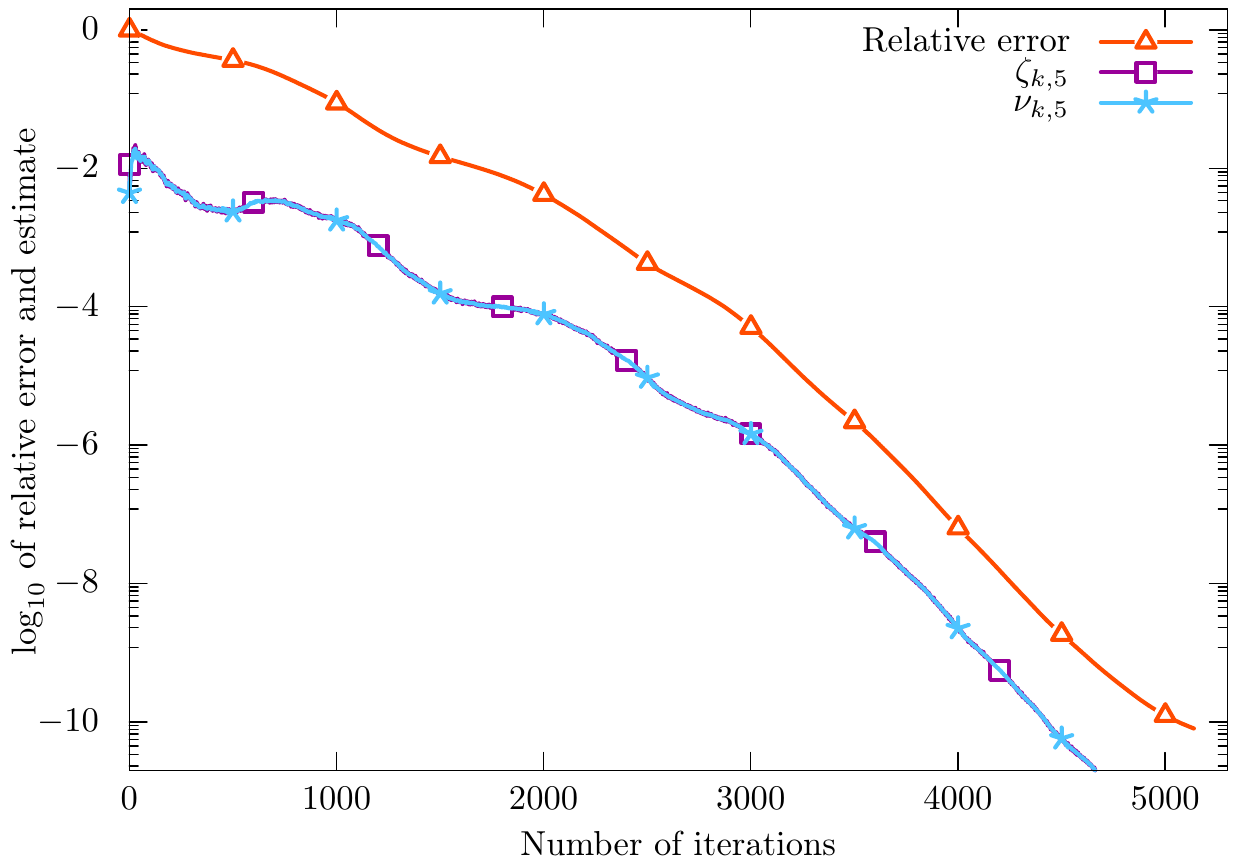}
		\subcaption{CurlCurl\_3.}
		\label{fig:CurlCurl_3_e}
	\end{minipage}
	\begin{minipage}{0.49\hsize}
		\centering\includegraphics[scale=0.6]{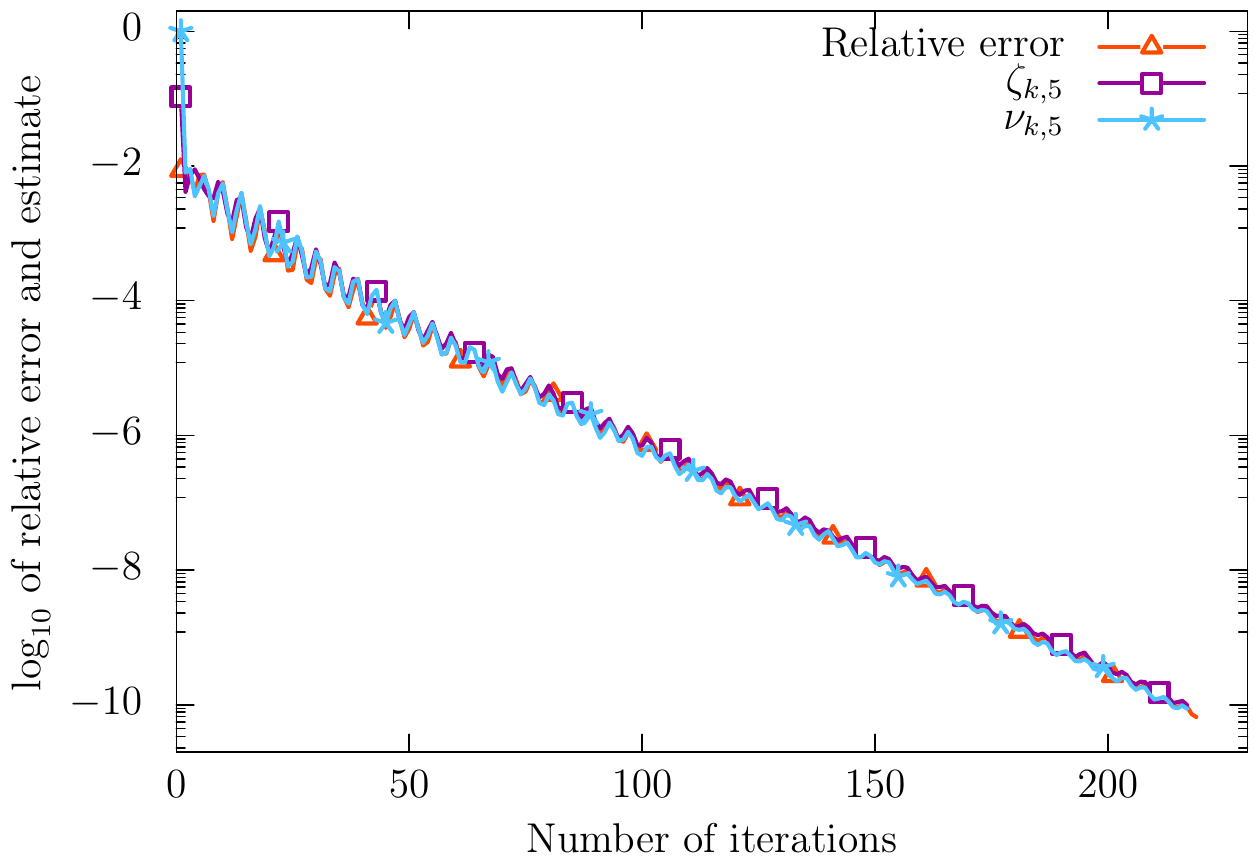}
		\subcaption{thermal2.}
		\label{fig:thermal2_e}
	\end{minipage}
	\caption{Relative error and its estimates $\mu_{k, d}$ and $\nu_{k, d}$ with $d = 5$ vs. number of iterations.}
	\label{fig:practical_e}
\end{figure}

\section{Conclusions} \label{sec:conc}
We explored the computation of quadratic forms of Hermitian matrix resolvents.
In contrast to previous shifted Krylov subspace methods, our method approximates the matrix resolvent directly.
The underlying concept used in approximating the resolvent is to exploit the moment-matching property of a shifted Lanczos method by solving the Vorobyev moment problem.
We showed that the shifted Lanczos method matches the first~$k$ moments of the original model and those of the reduced model and extended the scope of the problems that the standard Lanczos method can solve.
We derived the inverse of a linear operator representing the reduced model and related it to an entry of a Jacobi matrix resolvent.
The entry can be efficiently computed by using a recursive formula.
Previous shifted Krylov subspace methods work on a real symmetric matrix with a complex shift, whereas the proposed method works on a Hermitian matrix with a complex shift and does not break down, provided that the shift is not in the interior of the extremal eigenvalues of the Hermitian matrix.
We gave an error bound and estimates for the shifted Lanczos method for quadratic forms.
Numerical experiments on matrices drawn from real-world applications showed that the shifted Lanczos method is competitive with the shifted MINRES method and outperforms it when solving some problems.
We illustrated the effect of the spectrum and shift on the convergence and showed that the error estimate is reasonable.

We intend to perform preconditioning for the shifted Lanczos method for quadratic forms in future work.
The shifted COCG and COCR methods can use preconditioning if the preconditioned matrix is complex symmetric; however, for the shifted Lanczos method, it is not trivial to incorporate preconditioning.
How to monitor the convergence of the proposed method is also not trivial, and we leave the development of a more rigorous and/or sophisticated estimate for future work.
With regard to the error estimate, there will be an interesting connection with the Gauss quadratures (cf.\ \cite{GolubMeurant1993,GolubMeurant1997BIT,MeurantTichy2013NUMA}).

\section*{Acknowledgments}
The author would like to thank Editage for English language editing, Professor Zden\v{e}k Strako\v{s} for giving the author a chance to attend his course which motivated the author to study this work, and Professor Ken Hayami for helpful discussions.

\end{document}